\documentclass[11pt,a4paper,reqno]{amsart}
\usepackage[T1]{fontenc}
\usepackage[utf8]{inputenc}
\usepackage{lmodern}
\usepackage{amsmath,amssymb,mathtools}
\usepackage{microtype}
\usepackage[a4paper,left=23mm,right=23mm,top=23mm,bottom=25mm,headheight=13pt,headsep=8mm]{geometry}
\usepackage[hidelinks,bookmarksnumbered]{hyperref}
\hypersetup{pdftitle={Rankin--Cohen Pairings and Hilbert Hecke Eigenform Product Identities},pdfauthor={},pdfsubject={},pdfkeywords={Hilbert modular forms, Hecke eigenforms, Rankin--Cohen brackets, Rankin--Selberg L-functions}}
\numberwithin{equation}{section}
\newtheorem{theorem}{Theorem}[section]
\newtheorem{proposition}[theorem]{Proposition}
\newtheorem{lemma}[theorem]{Lemma}
\newtheorem{corollary}[theorem]{Corollary}
\newtheorem*{theoremA}{Theorem A}
\newtheorem*{theoremB}{Theorem B}

\newtheorem*{introtheorem}{Theorem}
\theoremstyle{remark}
\newtheorem*{remark}{Remark}
\DeclareMathOperator{\GL}{GL}
\DeclareMathOperator{\Cl}{Cl}
\DeclareMathOperator{\Tr}{Tr}
\DeclareMathOperator{\Ind}{Ind}
\DeclareMathOperator{\sgn}{sgn}
\DeclareMathOperator{\diag}{diag}
\renewcommand{\Re}{\operatorname{Re}}
\newcommand{\A}{\mathbb A}
\newcommand{\C}{\mathbb C}
\newcommand{\R}{\mathbb R}
\newcommand{\Q}{\mathbb Q}
\newcommand{\Z}{\mathbb Z}
\newcommand{\HH}{\mathbb H}
\newcommand{\OF}{\mathcal O_F}
\newcommand{\df}{\mathfrak d_F}
\newcommand{\pp}{\mathfrak p}
\newcommand{\ideal}{\mathfrak a}
\newcommand{\ttt}{\mathfrak t}
\newcommand{\one}{\mathbf 1}
\newcommand{\slashk}[1]{\mathbin{\|}_{#1}}
\newcommand{\Lcl}{L_{\mathrm{cl}}}
\newcommand{\Eeig}{E^{\mathrm{eig}}}
\newcommand{\CF}{\mathcal C_{F,k,\ell,r}}
\newcommand{\iotaH}{(i,\ldots,i)}
\newcommand{\smat}[4]{\left(\begin{smallmatrix}#1&#2\\#3&#4\end{smallmatrix}\right)}
\newcommand{\bmat}[4]{\begin{pmatrix}#1&#2\\#3&#4\end{pmatrix}}

\newcommand{\rtensor}{\mathop{\bigotimes\nolimits'}\limits}

\newcommand{\HQZtheoremRef}{\cite[Theorem~1]{HQZ}}
\newcommand{\CPSthmRef}{\cite[Theorem~2.4]{CPS}}
\newcommand{\ShahidiThmRef}{\cite[Theorem~5.2]{Shahidi}}
\allowdisplaybreaks[2]
\title{Rankin--Cohen Pairings and Hilbert Hecke Eigenform Product Identities}
\date{September 17, 2026}
\subjclass[2020]{Primary 11F41; Secondary 11F60, 11F66}
\keywords{Hilbert modular forms, Rankin--Cohen brackets, Hecke eigenforms, Rankin--Selberg $L$-functions}
\author{Jialin Li}
\address{School of Mathematics and Statistics, Wuhan University, Wuhan 430072, Hubei, China}
\email{jlli.math@whu.edu.cn}
\begin{document}

\begin{abstract}
We establish a Petersson--Rankin--Selberg identity for Hilbert Rankin--Cohen brackets over totally real fields of degree greater than one and arbitrary narrow class number. This extends results of Zhang--Zhang and Zhang--Zhou \cite{MZZ,ZZ}. For even $k,\ell\ge2$, $r\in\Z_{\ge0}$, $K=k+\ell+2r$, and normalized full-level cuspidal Hecke eigentuples $f\in S_K(\omega_f)$ and $h\in S_\ell(\omega_h)$, we show
\[
 \langle f,[E_k,h]_r\rangle
 =\CF
 \frac{L^S(k/2,\Pi(f)\times\Pi(h)^\vee)}
 {L_F(k,\omega_f\omega_h^{-1})}\ne0,
 \qquad \CF>0.
\]
Specializing to $k=2$, we remove the remaining GRH assumption in the product classification of Hao--Qin--Zhou \cite{HQZ}. Hence, over real quadratic fields of narrow class number one, the only full-level product identities among Hecke eigenforms of even parallel weights at least two, up to interchanging the factors, are
\[
 \Eeig_4=60(\Eeig_2)^2,\qquad h_8=120\Eeig_2h_6
 \qquad\text{over }\Q(\sqrt5).
\]
\end{abstract}
\maketitle

\section{Introduction and main results}\label{sec:introduction}
Let $F$ be a totally real field of degree $d>1$, and let $h_F^+$ be its narrow class number. Write $D_F$ for the absolute discriminant and $\zeta_F$ for the Dedekind zeta function. For even $m\ge2$, let $M_m$ and $S_m$ be the full-level tuple spaces of Hilbert modular forms and cusp forms of parallel weight $m$ and trivial classical level character.

Following Shimura's componentwise notation \cite[Section~2]{Shimura1978}, a full-level form over a field with arbitrary narrow class number is a tuple. More precisely,
\[
 M_m\simeq\prod_{\lambda=1}^{h_F^+}M_m(\Gamma_\lambda),
 \qquad
 S_m\simeq\prod_{\lambda=1}^{h_F^+}S_m(\Gamma_\lambda).
\]
The corresponding adelic spaces decompose according to the global central character,
\[
 M_m=\bigoplus_\omega M_m(\omega),
 \qquad
 S_m=\bigoplus_\omega S_m(\omega),
\]
where $\omega$ runs through the narrow ideal class characters for which the corresponding subspace is nonzero; see Joshi--Zhang \cite[Section~3]{JZ}. The Hecke algebra is commutative and normal and preserves every $M_m(\omega)$ and $S_m(\omega)$. We call a nonzero $g\in M_m(\omega_g)$ a simultaneous Hecke eigentuple if
\[
 T_vg=a_v(g)g\qquad\text{for every finite place }v,
\]
and normalize it by $c(\OF,g)=1$. When $h_F^+=1$, the central character is trivial and this is the usual definition of a normalized classical Hecke eigenform.

For $g\in M_k$, $h\in M_\ell$, and $r\ge0$, the Rankin--Cohen bracket satisfies
\[
 [g,h]_r\in M_{k+\ell+2r},\qquad [g,h]_0=gh.
\]
If one input is cuspidal, then the bracket is cuspidal. We ask when $[g,h]_r$ is again a Hecke eigentuple; for $r=0$, this is precisely the Hecke product problem. By a product identity we mean
\[
 P=cGH,
\]
where $P,G,H$ are nonzero normalized full-level Hecke eigentuples, $c\in\C^\times$, and the weight of $P$ is the sum of the weights of $G$ and $H$.

The product problem has been studied through Petersson pairings with Rankin--Cohen brackets. Joshi and Y.~Zhang proved that only finitely many real quadratic fields and weights can support full-level product identities and found the two identities over $\Q(\sqrt5)$ \cite[Theorem~7.4]{JZ}. M.~Zhang and Y.~Zhang proved the Petersson formula for arbitrary narrow class number when the Eisenstein weight is greater than two \cite[Proposition~4.4 and Corollary~4.5]{MZZ}. Y. Zhang and Y. Zhou established the Rankin--Selberg argument leading to the non-eigenform criterion for even $k\ge4$ \cite[Lemma~3.2 and Proposition~3.3]{ZZ}. They further stated the corresponding weight-two conclusion under GRH \cite[Remark~3.4]{ZZ}. Hao--Qin--Zhou subsequently obtained the full product classification under GRH \cite[Theorem~1]{HQZ}.

We briefly recall the argument of Zhang--Zhou in the narrow-class-number-one case. For even $k\ge4$, their Rankin--Selberg calculation shows that, for normalized Hecke eigenforms $f\in S_K$ and $h\in S_\ell$, \[ \langle f,[E_k,h]_r\rangle>0. \] Consequently, if $\dim S_K>1$ and $[E_k,h]_r$ were a Hecke eigenform, one could choose a normalized Hecke eigenform $f\in S_K$ orthogonal to $[E_k,h]_r$, giving a contradiction. Thus $[E_k,h]_r$ is not a Hecke eigenform when $\dim S_K>1$. Zhang--Zhou stated the corresponding conclusion for $k=2$ under GRH.
Hao--Qin--Zhou then applied this conclusion to obtain the following
classification of Hecke eigenform product identities.

\begin{introtheorem}[Hao--Qin--Zhou \HQZtheoremRef]
Among normalized full-level Hilbert Hecke eigenforms of even parallel weights at least two over real quadratic fields with $h_F^+=1$, and assuming GRH, the only product identities are
\[
 \Eeig_4=60(\Eeig_2)^2,\qquad h_8=120\Eeig_2h_6
 \qquad\text{over }\Q(\sqrt5),
\]
where $h_6$ and $h_8$ are the unique normalized cusp eigenforms of weights $6$ and $8$.
\end{introtheorem}

For $f\in S_K$ and $h\in S_\ell$, write
\[
 \Lcl(s;f,h)
 =
 \sum_{0\ne\ideal\subset\OF}
 \frac{c(\ideal,f)\overline{c(\ideal,h)}}{N(\ideal)^s}
\]
for the classical diagonal Rankin--Selberg series. For a narrow ideal class character $\chi$, put
\[
 L_F(s,\chi)=\prod_{0\ne\pp\subset\OF}
 \bigl(1-\chi(\pp)N(\pp)^{-s}\bigr)^{-1}
 \qquad(\Re(s)>1).
\]
This is the finite Hecke $L$-function. In particular, if $\omega_f=\omega_h$, then $L_F(s,\omega_f\omega_h^{-1})=\zeta_F(s)$.
Put
\[
 \CF
 =
 D_F^{-1/2}
 \left(
   \binom{k+r-1}{r}
   \frac{(K-2)!}{(4\pi)^{K-1}}
 \right)^d
 >0.
\]

Our first main result extends the Petersson identity to arbitrary
narrow class number and treats the boundary case of Eisenstein
weight two.

\begin{theoremA}
Let $F$ be totally real of degree $d>1$ and arbitrary narrow class
number. Let $k,\ell\ge2$ be even, $r\in\Z_{\ge0}$, and
$K=k+\ell+2r$. Let $E_k(\cdot,u)$ be the Eisenstein family defined
in \eqref{eq:eisenstein-family}. For $f\in S_K$ and $h\in S_\ell$,
\[
 \langle f,[E_k(\cdot,\overline u),h]_r\rangle
 =
 D_F^{-1/2}G_{k,\ell,r}(u)^d
 \Lcl(K-r-1+u;f,h)
 \qquad
 \left(\Re(u)>1-\frac{k}{2}\right).
\]
For $k=2$, this identity extends to the boundary value $u=0$.
If $f\in S_K(\omega_f)$ and $h\in S_\ell(\omega_h)$ are normalized full-level Hecke eigentuples, then
\[
 \langle f,[E_k,h]_r\rangle
 =
 \CF
 \frac{L^S(k/2,\Pi(f)\times\Pi(h)^\vee)}
      {L_F(k,\omega_f\omega_h^{-1})}
 \ne0.
\]
\end{theoremA}

Here $S$ is the set of real places and $\Pi(h)^\vee$ is the contragredient of $\Pi(h)$.

We sketch the proof of Theorem~A.  The first identity is obtained by
unfolding the Eisenstein family in the half-plane
$\Re(u)>1-k/2$.  For $k>2$, the point $u=0$ lies in this region.
When $k=2$, however, it lies on the boundary.  Uniform estimates
from Shimura's Fourier expansion, together with the rapid decay of
$h$, give
\[
 [E_2(\cdot,\overline u),h]_r
 \longrightarrow
 [E_2,h]_r
 \qquad\text{in Petersson }L^2,
\]
which allows us to pass to the limit on the Petersson side.

We then compare the classical diagonal series with the automorphic
Rankin--Selberg $L$-function.  The classical and adelic Hecke eigenvalues agree. If $f\in S_K(\omega_f)$ and $h\in S_\ell(\omega_h)$ are normalized eigentuples, the local Euler factors give
\[
 L_F(2w,\omega_f\omega_h^{-1})
 \Lcl\left(\frac{K+\ell-2}{2}+w;f,h\right)
 =
 L^S(w,\Pi(f)\times\Pi(h)^\vee).
\]
For $k=2$, the required value is $w=1$.
Proposition~\ref{prop:analytic-application}, together with the
results of Cogdell--Piatetski-Shapiro and Shahidi, gives
\[
 L^S(1,\Pi(f)\times\Pi(h)^\vee)\ne0.
\]
This proves the nonvanishing in Theorem~A.

The orthogonality argument is most naturally stated after projecting to a fixed central character. If $h\in S_\ell(\omega_h)$ and $P_\chi$ denotes the central-character projector defined in Subsection~\ref{subsec:dictionary}, then
\[
 \dim S_{k+\ell+2r}(\chi\omega_h)>1
 \quad\Longrightarrow\quad
 [P_\chi E_k,h]_r\text{ is not a Hecke eigentuple}.
\]
In particular, when $h_F^+=1$ there is only the trivial central character, and we recover
\[
 \dim S_{k+\ell+2r}>1
 \quad\Longrightarrow\quad
 [E_k,h]_r\text{ is not a Hecke eigenform}.
\]
For $r=0$, this gives the unconditional non-eigenform criterion needed in the product problem. We therefore obtain the following classification, stated again as Theorem~\ref{thm:classification}.
\begin{theoremB}
Among normalized full-level Hilbert Hecke eigenforms of even parallel weights at least two over real quadratic fields with $h_F^+=1$, the only product identities are
\[
 \Eeig_4=60(\Eeig_2)^2,\qquad h_8=120\Eeig_2h_6
 \qquad\text{over }\Q(\sqrt5).
\]
Here all eigenforms are normalized by $c(\OF,g)=1$, and $h_6$ and $h_8$ are the unique normalized cusp eigenforms of weights $6$ and $8$.
\end{theoremB}
We emphasize the normalization of the Eisenstein series appearing above. The $E_k$ in Theorem~A is Poincar\'e-normalized, whereas $E_k^{\mathrm{eig}}$ in Theorem~B denotes the Hecke-normalized Eisenstein eigentuple with $c(\OF,E_k^{\mathrm{eig}})=1.$  When $h_F^+=1$, the two normalizations are related by $ E_k^{\mathrm{eig}} = 2^{-d}\zeta_F(1-k)E_k. $

The paper is organized as follows. Section~\ref{sec:preliminaries} recalls the classical and adelic descriptions of Hilbert modular forms and the Rankin--Selberg results used below. Section~\ref{sec:proofs} proves the Petersson identity, including the case of Eisenstein weight two, and then applies the nonvanishing result to the product problem.

\section{Preliminaries}\label{sec:preliminaries}
\subsection{Classical and adelic Hilbert modular forms}\label{subsec:classical}\label{subsec:dictionary}
Let $F$ be totally real of degree $d>1$, with ring of integers $\OF$, different $\df$, and absolute discriminant $D_F$. Write $\Sigma_\infty$ for the set of real embeddings,
\[
 F_\infty=F\otimes_\Q\R\simeq\R^d,\qquad F_{\infty,+}^\times=(\R_{>0})^d,
\]
$\A_F$ for the adele ring, and $\A_F^\times$ for its idele group. A hat denotes finite profinite completion; thus $\widehat{\OF}=\prod_{v\nmid\infty}\mathcal O_v$. The narrow class group is
\[
 \Cl^+(F)=F^\times\backslash\A_F^\times/
 \bigl(F_{\infty,+}^\times\widehat{\OF}^{\,\times}\bigr),
 \qquad h_F^+=|\Cl^+(F)|.
\]
Choose finite ideles $t_\lambda$, $1\le\lambda\le h_F^+$, representing $\Cl^+(F)$, with $t_1=1$, and put
\[
 \ttt_\lambda=t_\lambda\widehat{\OF}\cap F,
 \qquad q_\lambda=|t_\lambda|_{\A_F}=N(\ttt_\lambda)^{-1}.
\]
Here $N$ is the absolute ideal norm and $|\cdot|_{\A_F}$ is the idelic norm. We write $\xi\gg0$ when $\tau(\xi)>0$ for every $\tau\in\Sigma_\infty$, and $\GL_2^+(F)$ for the matrices whose determinant is totally positive. We also write $\mathcal O_F^{\times,+}$ for the totally positive units and
\[
 \zeta_F(s)=\prod_{0\ne\pp\subset\OF}(1-N(\pp)^{-s})^{-1}
 \qquad(\Re(s)>1)
\]
for the Dedekind zeta function, continued meromorphically elsewhere. At full level set
\begin{equation*}
 \Gamma_\lambda=\left\{\bmat abcd\in\GL_2^+(F):
 \begin{array}{l}a,d\in\OF,\ b\in\ttt_\lambda^{-1}\df^{-1},\\
 c\in\ttt_\lambda\df,\ \det\in\mathcal O_F^{\times,+}
 \end{array}\right\}.
\end{equation*}
For $\gamma=(\gamma_\tau)_\tau\in\GL_2^+(F_\infty)$ and parallel weight $k$,
\[
 (g\slashk{k}\gamma)(z)=\prod_{\tau\in\Sigma_\infty}
 \det(\gamma_\tau)^{k/2}(c_\tau z_\tau+d_\tau)^{-k}g(\gamma z).
\]
A matrix $\gamma_\tau=\smat{a_\tau}{b_\tau}{c_\tau}{d_\tau}$ acts on $z_\tau\in\HH$ by $(a_\tau z_\tau+b_\tau)/(c_\tau z_\tau+d_\tau)$. We denote by $M_k(\Gamma_\lambda)$ the holomorphic forms invariant under this slash action and holomorphic at every cusp, and by $S_k(\Gamma_\lambda)$ the subspace whose constant term at every cusp is zero. Such a form is called cuspidal.

For a full-level form $g$, write its components as
\[
 g=(g_\lambda)_{\lambda=1}^{h_F^+},\qquad g_\lambda\in M_k(\Gamma_\lambda).
\]
Its expansion at the distinguished cusp is
\begin{equation*}
 g_\lambda(z)=a_\lambda(0;g)+
 \sum_{\substack{\xi\in\ttt_\lambda\\\xi\gg0}}a_\lambda(\xi;g)e^{2\pi i\Tr(\xi z)}.
\end{equation*}
Every nonzero integral ideal has a unique expression
\[
 \ideal=(\xi)\ttt_\lambda^{-1},\qquad \xi\in\ttt_\lambda,\qquad\xi\gg0,
\]
with $\lambda$ unique and $\xi$ unique modulo $\mathcal O_F^{\times,+}$. We use the classical Hecke normalization
\begin{equation*}
 c(\ideal,g)=q_\lambda^{k/2}a_\lambda(\xi;g).
\end{equation*}
This is well defined: if $\xi$ is replaced by $\varepsilon\xi$ with $\varepsilon\in\mathcal O_F^{\times,+}$, invariance under $\diag(\varepsilon,1)\in\Gamma_\lambda$ and $N(\varepsilon)=1$ gives $a_\lambda(\varepsilon\xi;g)=a_\lambda(\xi;g)$. Thus the right-hand side depends only on $\ideal$, not on its representative.

For $z_\tau=x_\tau+iy_\tau$, put
\[
 Y(z)=\prod_\tau y_\tau,\qquad
 d\mu(z)=\prod_\tau\frac{dx_\tau\,dy_\tau}{y_\tau^2}.
\]
The Petersson pairing, linear in the first variable, is
\begin{equation}\label{eq:petersson}
 \langle f,g\rangle=\sum_{\lambda=1}^{h_F^+}
 \int_{\Gamma_\lambda\backslash\HH^d} f_\lambda(z)\overline{g_\lambda(z)}Y(z)^k\,d\mu(z).
\end{equation}
The sum is not divided by $h_F^+$. We use this unaveraged normalization throughout; no class-number factor occurs below.

For $\tau\in\Sigma_\infty$ define
\[
 D_\tau=\frac1{2\pi i}\frac\partial{\partial z_\tau},\qquad
 D^t=\prod_\tau D_\tau^{t_\tau},\qquad |t|=\sum_\tau t_\tau.
\]
Here $t=(t_\tau)_\tau$ is a multi-index; in such expressions $\one$ denotes the all-one vector. For $g\in M_k$, $h\in M_\ell$, and $r\ge0$, the Rankin--Cohen bracket of order $r$ is
\begin{equation*}
 [g,h]_r=\sum_{t\in\{0,\ldots,r\}^{\Sigma_\infty}}(-1)^{|t|}
 \prod_\tau\binom{k+r-1}{r-t_\tau}\binom{\ell+r-1}{t_\tau}
 (D^tg)(D^{r\one-t}h).
\end{equation*}
The operator is applied componentwise, and
\[
 [g,h]_r\in M_{k+\ell+2r},\qquad[g,h]_r=(-1)^{dr}[h,g]_r.
\]
It is cuspidal if $r>0$, and also for $r=0$ when one input is cuspidal \cite{CKR,Zagier}.

We now pass from this componentwise classical description to its adelic realization and Hecke action.

For a finite place $v$, let $F_v$ and $\mathcal O_v$ be the local field and its ring of integers, and let $\pp_v=\varpi_v\mathcal O_v$ be its maximal ideal. We also write $\varpi_v$ for the finite idele whose $v$-component
is $\varpi_v$ and whose components away from $v$ are $1$. Put
$
 q_v=\#(\mathcal O_v/\pp_v)=N(\pp_v).
$ Choose $d_v\in F_v^\times$ with $\df\mathcal O_v=d_v\mathcal O_v$ and set
\[
 K_v^\circ=\GL_2(\mathcal O_v),\qquad
 t_v=\bmat{d_v^{-1}}001,\qquad
 K_v=t_vK_v^\circ t_v^{-1},\qquad
 K_f=\prod_{v\nmid\infty}K_v.
\]
Set the local double coset
\[
 T_v
 =
 K_v
 \begin{pmatrix}\varpi_v&0\\0&1\end{pmatrix}
 K_v.
\]
Conjugating the standard spherical double-coset decomposition by $t_v$ gives
\begin{equation}\label{eq:right-cosets}
 T_v=s_{v,\infty}K_v\ \sqcup\!
 \coprod_{u\in\mathcal O_v/\pp_v}s_{v,u}K_v,
 \qquad
 s_{v,\infty}=\bmat100{\varpi_v},\quad
 s_{v,u}=\bmat{\varpi_v}{d_v^{-1}u}01.
\end{equation}
We use the same symbols for their embeddings in $\GL_2(\A_F)$, with all components away from $v$ equal to $I_2$.

When $h_F^+>1$, the adelic quotient has one determinant component for each narrow ideal class. Classically, a full-level form is therefore represented by the tuple above. In general, we write
\begin{equation}\label{eq:tuple}
 M_k\simeq\prod_{\lambda=1}^{h_F^+}M_k(\Gamma_\lambda),
 \qquad S_k\simeq\prod_{\lambda=1}^{h_F^+}S_k(\Gamma_\lambda).
\end{equation}
This is the full tuple space; at this stage no global central character is prescribed. The Hecke operators may move between its components.

With $x_\lambda=\diag(t_\lambda^{-1},1)$, strong approximation gives
\begin{equation}\label{eq:components}
 \GL_2(\A_F)=\coprod_{\lambda=1}^{h_F^+}
 \GL_2(F)x_\lambda\GL_2^+(F_\infty)K_f,
 \qquad\Gamma_\lambda=\GL_2^+(F)\cap x_\lambda K_fx_\lambda^{-1},
\end{equation}
where the intersection is taken in the finite adeles. The adelization of $g=(g_\lambda)_\lambda$ is the function $\Phi_g:\GL_2(\A_F)\to\C$ defined by
\begin{equation}\label{eq:adelization}
 \Phi_g(\gamma x_\lambda g_\infty k_f)
 =(g_\lambda\slashk{k}g_\infty)\iotaH,
 \qquad
 \gamma\in\GL_2(F),\ g_\infty\in\GL_2^+(F_\infty),\ k_f\in K_f.
\end{equation}
The decomposition \eqref{eq:components} and the adelization \eqref{eq:adelization} give the adelic realization of the classical tuple decomposition \eqref{eq:tuple}; see Garrett \cite[Sections~3.1--3.4]{Garrett} and Raghuram--Tanabe \cite[Sections~4.1.1--4.1.4]{RT}.

In particular, $\Phi_g$ is left $\GL_2(F)$-invariant and right $K_f$-invariant. We set $(R(a)\Phi)(x)=\Phi(xa)$. Then
\begin{equation}\label{eq:adelic-hecke}
 T_v\Phi=R(s_{v,\infty})\Phi
 +\sum_{u\in\mathcal O_v/\pp_v}R(s_{v,u})\Phi.
\end{equation}

Let $Z(\A_F)=\{zI_2:z\in\A_F^\times\}$. Suppose that $g$ is an eigenvector for the central action,
\[
 \Phi_g(zx)=\omega(z)\Phi_g(x)
 \qquad(z\in\A_F^\times).
\]
Then $\omega$ is a narrow ideal class character. Indeed, left $\GL_2(F)$-invariance gives $\omega(a)=1$ for $a\in F^\times$. If $u\in\widehat{\OF}^{\,\times}$, then $uI_2\in K_f$, so right $K_f$-invariance gives $\omega(u)=1$. Finally, if $a_\infty=(a_\tau)_\tau\in F_{\infty,+}^\times$, then the scalar slash factor is
\[
 \prod_{\tau\in\Sigma_\infty}
 \det(a_\tau I_2)^{k/2}a_\tau^{-k}=1,
\]
so $\omega(a_\infty)=1$. Hence $\omega$ factors through
\[
 F^\times\backslash\A_F^\times/
 \bigl(F_{\infty,+}^\times\widehat{\OF}^{\,\times}\bigr)
 =\Cl^+(F).
\]
In particular, $\omega$ has finite order and is unitary. For such a character define
\[
 M_k(\omega)=\{g\in M_k:\Phi_g(zx)=\omega(z)\Phi_g(x)\},
 \qquad S_k(\omega)=S_k\cap M_k(\omega).
\]
On $S_k$, the central action is unitary for the Petersson pairing, so distinct central-character subspaces are orthogonal. Joshi--Zhang \cite[Section~3]{JZ} give the decomposition
\begin{equation}\label{eq:central-decomposition}
 M_k=\bigoplus_\omega M_k(\omega),
 \qquad
 S_k=\bigoplus_\omega S_k(\omega),
\end{equation}
and show that the full-level Hecke algebra is commutative and normal and preserves each summand.

For later use, let $P_\chi$ be the central-character projector on $M_k$. Its restriction to $S_k$ is orthogonal for the Petersson pairing. With the chosen representatives $t_\lambda$, it is characterized adelically by
\begin{equation}\label{eq:central-projector}
 \Phi_{P_\chi g}
 =\frac1{h_F^+}\sum_{\lambda=1}^{h_F^+}
 \chi(t_\lambda)^{-1}R(t_\lambda I_2)\Phi_g.
\end{equation}
The definition is independent of the chosen class representatives. Since scalar matrices are central,
\begin{equation}\label{eq:projector-hecke-commute}
 P_\chi T_v=T_vP_\chi\qquad(v\nmid\infty).
\end{equation}
Indeed, every $R(t_\lambda I_2)$ commutes with every right translation in \eqref{eq:adelic-hecke}. Consequently, if $F$ is a simultaneous Hecke eigentuple, then $P_\chi F$ is either zero or a simultaneous Hecke eigentuple with the same $T_v$-eigenvalues. If $h\in M_\ell(\omega_h)$, central covariance of the Rankin--Cohen bracket also gives
\begin{equation}\label{eq:projector-bracket}
 P_{\chi\omega_h}[g,h]_r=[P_\chi g,h]_r.
\end{equation}

We next return to the classical tuple and make the action of $T_v$ on its components and Fourier coefficients explicit. For $\lambda$ and $v$, define $\lambda[v]$ by
\[
 [\ttt_{\lambda[v]}]=[\pp_v^{-1}\ttt_\lambda]\qquad\text{in }\Cl^+(F).
\]
Applying \eqref{eq:components} to $x_\lambda s_{v,u}$ and $x_\lambda s_{v,\infty}$ shows that both lie in the component $\lambda[v]$. Choose $r_{\lambda,v,u},r_{\lambda,\infty}\in\GL_2^+(F)$ so that
\begin{equation}\label{eq:component-transport}
 \begin{split}
 x_\lambda s_{v,u}&=r_{\lambda,v,u}^{-1}x_{\lambda[v]}(r_{\lambda,v,u})_\infty k_{f,u},\\
 x_\lambda s_{v,\infty}&=r_{\lambda,\infty}^{-1}x_{\lambda[v]}(r_{\lambda,\infty})_\infty k_{f,\infty},
 \end{split}
\end{equation}
with $k_{f,u},k_{f,\infty}\in K_f$. We define
\begin{equation}\label{eq:tuple-hecke}
 (T_vg)_\lambda
 =g_{\lambda[v]}\slashk{k}r_{\lambda,\infty}
 +\sum_{u\in\mathcal O_v/\pp_v}g_{\lambda[v]}\slashk{k}r_{\lambda,v,u}.
\end{equation}
This is the standard Hecke action on narrow-class tuples; see \cite[Section~3.2]{Garrett}, \cite[Section~2, p.~648]{Shimura1978}, and \cite[Sections~4.1.3--4.1.4]{RT}. Different choices of the $r$'s change the corresponding summands by elements of $\Gamma_{\lambda[v]}$, so \eqref{eq:tuple-hecke} is well defined.

We briefly check the Fourier coefficients. Put $\mu=\lambda[v]$. Choose $\alpha\in F^\times$ with
$\alpha\gg0$ and $\varepsilon\in\widehat{\OF}^{\,\times}$ with
\[
  t_\mu=\alpha t_\lambda\varpi_v^{-1}\varepsilon,
  \qquad\text{equivalently}\qquad
  \mathfrak t_\mu=\alpha\pp_v^{-1}\mathfrak t_\lambda,
\]
and choose $\beta_u\in F$ such that
\[
  t_\mu\beta_u
  \equiv-\varepsilon\varpi_v^{-1}d_v^{-1}u
  \pmod{\widehat{\mathfrak d}_F^{-1}}.
\]Then we may take
\[
 r_{\lambda,v,u}=\begin{pmatrix}\alpha^{-1}&\beta_u\\0&1\end{pmatrix},
 \qquad x_\mu^{-1}r_{\lambda,v,u}x_\lambda\in K_fs_{v,u}^{-1}.
\]
Thus
\[
 (g_\mu\slashk{k}r_{\lambda,v,u})(z)
 =N(\alpha)^{-k/2}\sum_\eta a_\mu(\eta;g)
 e^{2\pi i\Tr(\eta\beta_u)}e^{2\pi i\Tr(\alpha^{-1}\eta z)}.
\]
If $\ideal=(\xi)\mathfrak t_\lambda^{-1}$, its $\xi$-coefficient comes from $\eta=\alpha\xi$; moreover $(\alpha\xi)\mathfrak t_\mu^{-1}=\pp_v\ideal$ and the congruence for $\beta_u$ gives $e^{2\pi i\Tr(\alpha\xi\beta_u)}=1$. Summing over $u$ therefore gives $q_v^{1-k/2}c(\pp_v\ideal,g)$. Suppose now that $g\in M_k(\omega_g)$. For $s_{v,\infty}$, the identity
\[
 s_{v,\infty}=(\varpi_vI_2)s_{v,0}^{-1}
\]
contributes the central scalar $\omega_g(\varpi_v)=\omega_g(\pp_v)$. Hence
\begin{equation}\label{eq:hecke-coefficients}
 \begin{split}
 c(\ideal,T_vg)
 &=q_v^{1-k/2}c(\pp_v\ideal,g)
   +\omega_g(\pp_v)q_v^{k/2}c(\pp_v^{-1}\ideal,g),\\
 c(\pp_v^{-1}\ideal,g)&=0\quad\text{if }\pp_v^{-1}\ideal\not\subset\OF.
 \end{split}
\end{equation}
If $g$ is a simultaneous Hecke eigentuple normalized by $c(\OF,g)=1$, then
\begin{equation}\label{eq:hecke-eigenvalue}
 a_v(g)=q_v^{1-k/2}c(\pp_v,g),
 \qquad T_vg=q_v^{1-k/2}c(\pp_v,g)g.
\end{equation}
Equivalently, the prime-power recurrence begins with
\begin{equation}\label{eq:hecke-square}
 c(\pp_v^2,g)=c(\pp_v,g)^2-\omega_g(\pp_v)q_v^{k-1}.
\end{equation}
This is the standard full-level Hecke recurrence with central character; compare Joshi--Zhang \cite[Section~3, (3.2)]{JZ}. Our $T_v$ is the unitary normalization $q_v^{1-k/2}T_{\pp_v}^{\mathrm{cl}}$ of their classical Hecke operator.

A nonzero tuple $g\in M_k(\omega_g)$ is a simultaneous Hecke eigentuple if $T_vg=a_v(g)g$ for every finite $v$. When $h_F^+=1$, this is the usual notion of a classical simultaneous Hecke eigenform. The full-level Hecke algebra is commutative and normal for the Petersson pairing and preserves every central-character subspace \cite[Section~3]{JZ}. Hence $S_k$, and each $S_k(\omega)$, has an orthogonal basis of simultaneous Hecke eigentuples. For nontrivial $\omega_g$, the normalized Fourier coefficients need not be real.

By construction, the classical and adelic Hecke actions are compatible:
\begin{equation}\label{eq:hecke-intertwining}
 \Phi_{T_vg}=T_v\Phi_g,
 \qquad
 T_vg=a_v(g)g\quad\Longleftrightarrow\quad T_v\Phi_g=a_v(g)\Phi_g.
\end{equation}
Thus a classical Hecke eigentuple and its adelization have the same local Hecke eigenvalues.

We recall only the automorphic-representation facts used below. An idele-class quasi-character is a continuous homomorphism
\[
 \eta:F^\times\backslash\A_F^\times\longrightarrow\C^\times.
\]
Following Jacquet--Langlands, let $A_0(\eta)$ be the right-regular representation on cuspidal automorphic forms $\varphi$ with central quasi-character $\eta$ \cite[Definition~10.2 and Section~9]{JL}. Thus $\varphi(zx)=\eta(z)\varphi(x)$. The cuspidality means
\begin{equation}\label{eq:cuspidality}
 \int_{F\backslash\A_F}\varphi\left(\bmat1x01a\right)\,dx=0
 \qquad(a\in\GL_2(\A_F)).
\end{equation}
By \cite[Proposition~10.9 and Section~11]{JL}, this representation has the multiplicity-free algebraic decomposition
\begin{equation*}
 A_0(\eta)=\bigoplus_\Pi\Pi,
\end{equation*}
where each $\Pi$ is an irreducible admissible cuspidal automorphic representation.

If $g\in S_k(\omega_g)$, then its adelization belongs to $A_0(\omega_g)$. If $g$ is also a normalized Hecke eigentuple, the multiplicity-one and strong-multiplicity-one argument in \cite[Theorem~4.7]{RT} places $\Phi_g$ in a unique summand, denoted by $\Pi(g)$. Thus
\begin{equation*}
 \Phi_g\in\Pi(g)\subset A_0(\omega_g).
\end{equation*}
This is the irreducible cuspidal automorphic representation attached to $g$.

Fix the nontrivial unitary additive character
\begin{equation*}
 \psi=\psi_\Q\circ\Tr_{F/\Q}:F\backslash\A_F\longrightarrow\C^\times,
 \qquad\psi_{\Q,\infty}(x)=e^{2\pi ix},\qquad\psi=\prod_v\psi_v.
\end{equation*}
For $\varphi\in\Pi$, define
\begin{equation}\label{eq:whittaker}
 W_\varphi(a)=\int_{F\backslash\A_F}\varphi\left(\bmat1x01a\right)\overline{\psi(x)}\,dx,
 \qquad W(\Pi,\psi)=\{W_\varphi:\varphi\in\Pi\}.
\end{equation}
The global Whittaker-model theorem gives the equivariant isomorphism
\begin{equation*}
 J_\psi:\Pi\xrightarrow{\sim}W(\Pi,\psi),\qquad\varphi\longmapsto W_\varphi;
\end{equation*}
see \cite[Theorem~2.2]{RT} and \cite[Section~11]{JL}. In particular, every cuspidal representation of $\GL_2(\A_F)$ is globally generic. The tensor-product theorem gives
\begin{equation*}
 \Pi\simeq\rtensor_v\Pi_v,
\end{equation*}
where each $\Pi_v$ is an irreducible admissible representation of $\GL_2(F_v)$ \cite[Proposition~9.1]{JL}; see also \cite[Theorem~4.9]{RT}.

We describe the two local representations that occur. For $m\ge1$, let $D_m^\infty$ be the holomorphic discrete-series representation of $\GL_2(\R)$ used in \cite[Section~3.1.3 and Theorem~4.7]{RT}. Its central character is $a\mapsto\sgn(a)^{m+1}$, and its least nonnegative rotation exponent is $m+1$: for
\[
 r_\theta=\bmat{\cos\theta}{-\sin\theta}{\sin\theta}{\cos\theta}
\]
there is a lowest-weight vector $v_0\ne0$ with
\[
 D_m^\infty(r_\theta)v_0=e^{-i(m+1)\theta}v_0.
\]
Since $|\det r_\theta|^{it}=1$, twisting does not change the rotation exponents. Hence
\[
 D_i^\infty\not\simeq D_j^\infty\otimes|\det|^{it}\qquad(i\ne j,\ t\in\R).
\]
Let $B_v$ be the upper triangular subgroup of $\GL_2(F_v)$. For smooth characters $\chi_1,\chi_2:F_v^\times\to\C^\times$, the normalized induced representation
\[
 I_v(\chi_1,\chi_2)=\Ind_{B_v}^{\GL_2(F_v)}(\chi_1\otimes\chi_2)
\]
is the space of locally constant functions $\phi:\GL_2(F_v)\to\C$ satisfying
\[
 \phi\left(\bmat a*0d x\right)=\left|\frac ad\right|_v^{1/2}
 \chi_1(a)\chi_2(d)\phi(x).
\]
The group acts by right translation, $(I_v(\chi_1,\chi_2)(g)\phi)(x)=\phi(xg)$. A character of $F_v^\times$ is unramified when it is trivial on $\mathcal O_v^\times$. Since $F_v^\times=\varpi_v^\Z\mathcal O_v^\times$, such a character is determined by $\alpha=\chi(\varpi_v)$ and satisfies
\[
 \chi(\varpi_v^nu)=\alpha^n\qquad(n\in\Z,\ u\in\mathcal O_v^\times).
\]
The scalar matrix $zI_2$ acts by $\chi_1(z)\chi_2(z)$. When both characters are unramified and $K_v^\circ=\GL_2(\mathcal O_v)$,
\begin{equation*}
 \dim I_v(\chi_1,\chi_2)^{K_v^\circ}=1.
\end{equation*}
The vector $\phi_v^\circ$ in this line with $\phi_v^\circ(1)=1$ is the normalized spherical vector \cite[Chapter~I, Lemma~3.9]{JL}.

A constituent of a finite-length representation is an irreducible quotient in a composition series. An irreducible smooth representation $\pi_v$ of $\GL_2(F_v)$ is spherical, or unramified, when $\pi_v^{K_v^\circ}\ne0$. Among the irreducible constituents of $I_v(\chi_1,\chi_2)$ there is
exactly one spherical constituent. If $I_v(\chi_1,\chi_2)$ is
irreducible, this constituent is the whole induced representation.

Every irreducible unramified representation is obtained in this way from two unramified characters \cite[Chapter~I, Lemma~3.9]{JL}. The normalized principal series is irreducible precisely when
\begin{equation}\label{eq:irreducibility}
 \chi_1\chi_2^{-1}\ne|\cdot|_v^{\pm1}.
\end{equation}
In the exceptional case, after possibly interchanging the characters,
\[
 \chi_1=\chi|\cdot|_v^{1/2},\qquad\chi_2=\chi|\cdot|_v^{-1/2}.
\]
Then the spherical constituent is the one-dimensional representation $\chi\circ\det$; the other constituent is a twist of the Steinberg representation and has no $K_v^\circ$-fixed vector \cite[Chapter~I, Theorem~3.3]{JL}.

\begin{proposition}\label{prop:local-components}
Let $g\in S_k(\omega_g)$ be a normalized full-level cuspidal eigentuple of even parallel weight $k\ge2$. The representation $\Pi(g)$ is irreducible and has the restricted tensor decomposition
\begin{equation}\label{eq:attached-tensor}
 \Pi(g)\simeq\rtensor_v\Pi(g)_v.
\end{equation}
The restricted tensor product is taken with respect to the normalized spherical vector at almost every finite place. Moreover,
\begin{align*}
 \Pi(g)_\tau&\simeq D_{k-1}^\infty\qquad(\tau\mid\infty),\\
 \Pi(g)_v&\simeq I_v(\chi_{1,g,v},\chi_{2,g,v}),\quad
 \chi_{1,g,v}\chi_{2,g,v}=\omega_{g,v}\qquad(v\nmid\infty),
\end{align*}
where the two characters are unramified and the displayed principal series is irreducible. Writing $\alpha_{i,g,v}=\chi_{i,g,v}(\varpi_v)$, one has
\begin{equation}\label{eq:satake-product}
 \alpha_{1,g,v}\alpha_{2,g,v}=\omega_g(\pp_v),
 \qquad |\alpha_{1,g,v}|=|\alpha_{2,g,v}|=1.
\end{equation}
\end{proposition}
\begin{proof}
By definition, $\Pi(g)$ is an irreducible admissible summand of $A_0(\omega_g)$ containing $\Phi_g$. By \cite[Proposition~9.1]{JL} or \cite[Theorem~4.9]{RT}, we obtain \eqref{eq:attached-tensor}, with every $\Pi(g)_v$ irreducible and admissible.

At a real place, the definition of $\Phi_g$ gives
\[
 R(r_{\theta,\tau})\Phi_g=e^{-ik\theta}\Phi_g,\qquad
 \frac{\partial g_\lambda}{\partial\overline z_\tau}=0.
\]
Thus $\Phi_g$ has weight $k$ and is holomorphic at $\tau$. The real-place assertion of \cite[Theorem~4.7]{RT} gives $\Pi(g)_\tau\simeq D_{k-1}^\infty$.

It remains to treat a finite place $v$. Put
\[
 K_v^\circ=\GL_2(\mathcal O_v),\qquad t_v=\bmat{d_v^{-1}}001,
 \qquad K_v=t_vK_v^\circ t_v^{-1}.
\]
By definition, $\Phi_g$ is fixed by $K_v$. Viewing $t_v$ as an adele supported at $v$, the vector $R(t_v^{-1})\Phi_g$ is nonzero and $K_v^\circ$-fixed. Hence $\Pi(g)_v$ is the spherical constituent of $I_v(\chi_{1,g,v},\chi_{2,g,v})$ for two unramified characters. The scalar action gives
\[
 \chi_{1,g,v}\chi_{2,g,v}=\omega_{g,v},
\]
and therefore $\alpha_{1,g,v}\alpha_{2,g,v}=\omega_g(\pp_v)$. Since $\omega_g$ is a narrow ideal class character, it is unitary. The Ramanujan theorem \cite[Theorem~1 and Section~2.2]{Blasius} gives $|\alpha_{1,g,v}|=|\alpha_{2,g,v}|=1$. Thus both inducing characters are unitary, and \eqref{eq:irreducibility} gives irreducibility.
\end{proof}

\begin{lemma}\label{lem:satake-hecke}
Let $g$ be as in Proposition~\ref{prop:local-components}. For every finite place $v$,
\begin{equation}\label{eq:satake-hecke}
 \alpha_{1,g,v}+\alpha_{2,g,v}=q_v^{-(k-1)/2}c(\pp_v,g).
\end{equation}
\end{lemma}
\begin{proof}
Equations \eqref{eq:hecke-eigenvalue} and \eqref{eq:hecke-intertwining} give
\begin{equation}\label{eq:adelic-eigenvalue}
 T_v\Phi_g=q_v^{1-k/2}c(\pp_v,g)\Phi_g.
\end{equation}
By Proposition~\ref{prop:local-components}, this operator acts only on the $v$-factor. Let $t_v$, $K_v^\circ$, and $K_v$ be as in its proof. Since $K_v=t_vK_v^\circ t_v^{-1}$ and $t_v$ commutes with $\diag(\varpi_v,1)$, applying $R(t_v^{-1})$ to \eqref{eq:adelic-eigenvalue} gives
\begin{equation}\label{eq:spherical-action}
 \begin{split}
 R\left(\bmat100{\varpi_v}\right)R(t_v^{-1})\Phi_g
 &+\sum_{u\in\mathcal O_v/\pp_v}R\left(\bmat{\varpi_v}u01\right)R(t_v^{-1})\Phi_g\\
 &=q_v^{1-k/2}c(\pp_v,g)R(t_v^{-1})\Phi_g.
 \end{split}
\end{equation}
The vector $R(t_v^{-1})\Phi_g$ is nonzero and $K_v^\circ$-fixed, so its $v$-factor spans the spherical line generated by $\phi_v^\circ$. The normalized induction law gives $q_v^{1/2}\alpha_{2,g,v}$ from the first term and $q_v^{-1/2}\alpha_{1,g,v}$ from each of the $q_v$ terms in the sum. Hence the left side acts by $q_v^{1/2}(\alpha_{1,g,v}+\alpha_{2,g,v})$, and comparison with \eqref{eq:spherical-action} proves \eqref{eq:satake-hecke}.
\end{proof}

We also use the following terminology. A character $\eta$, global or local, is unitary if $|\eta(t)|=1$ for every $t$. The central character of an automorphic representation $\Pi$ is the
idele-class character by which the scalar matrices act.
It factors through $F^\times\backslash\A_F^\times$. The representation is cuspidal when \eqref{eq:cuspidality} holds for every $\varphi\in\Pi$, and it is unitary when it admits an invariant positive-definite Hermitian form. At a finite place, an irreducible representation is unramified when it has a nonzero $\GL_2(\mathcal O_v)$-fixed vector.

\begin{proposition}\label{prop:global-properties}
Let $g$ be as in Proposition~\ref{prop:local-components}. Then $\Pi(g)$ is an irreducible unitary cuspidal automorphic representation with central character $\omega_g$. It is globally generic and unramified at every finite place. Its contragredient $\Pi(g)^\vee$ has central character $\omega_g^{-1}$ and has the same cuspidality, unitarity, genericity, and finite-place unramifiedness properties. More precisely, at every finite place,
\[
 \Pi(g)_v^\vee\simeq I_v(\chi_{1,g,v}^{-1},\chi_{2,g,v}^{-1}).
\]
\end{proposition}
\begin{proof}
The adelic unipotent integral is the constant Fourier coefficient at the corresponding classical cusp. Since $g\in S_k$, these terms vanish, so $\Pi(g)$ is cuspidal \cite[Section~4.2 and Theorem~4.7]{RT}. By construction, $\Phi_g(zx)=\omega_g(z)\Phi_g(x)$, so the central character is $\omega_g$. This character is finite order and hence unitary. Therefore the usual invariant Petersson form on the quotient by the center is finite and positive definite on $\Pi(g)$; see \cite[Proposition~10.7]{JL}. Thus $\Pi(g)$ is unitary. Proposition~\ref{prop:local-components} gives unramifiedness at every finite place. Finally, the adelic Fourier expansion and \eqref{eq:whittaker} give a Whittaker coefficient of $\Phi_g$ that is a nonzero scalar multiple of $c(\OF,g)=1$, so $\Pi(g)$ is globally generic.

The contragredient of an irreducible unitary cuspidal representation is again irreducible, unitary, and cuspidal; its central character is the inverse, and global genericity passes from $\psi$ to $\psi^{-1}$. At a finite place, normalized induction gives
\[
 I_v(\chi_1,\chi_2)^\vee\simeq I_v(\chi_1^{-1},\chi_2^{-1}),
\]
which proves the final assertion.
\end{proof}

\subsection{Rankin--Selberg $L$-functions}\label{subsec:analytic}
We recall the standard $L$-function of Jacquet--Langlands and the Rankin--Selberg $L$-function of Jacquet--Piatetski-Shapiro--Shalika \cite{JL,JPSS}. Let $\Sigma_1$ and $\Sigma_2$ be irreducible unitary cuspidal automorphic representations of $\GL_2(\A_F)$. At a finite unramified place, write
\[
 \Sigma_{1,v}\simeq I_v(\chi_{1,v},\chi_{2,v}),\qquad
 \Sigma_{2,v}\simeq I_v(\mu_{1,v},\mu_{2,v}),
\]
and put $\alpha_{i,v}=\chi_{i,v}(\varpi_v)$ and $\beta_{j,v}=\mu_{j,v}(\varpi_v)$. Then
\begin{equation*}
 L_v(w,\Sigma_{1,v}\times\Sigma_{2,v})
 =\prod_{i,j=1}^2(1-\alpha_{i,v}\beta_{j,v}q_v^{-w})^{-1}.
\end{equation*}
For $\Sigma_{2,v}^\vee$, the factors contain $\alpha_{i,v}\beta_{j,v}^{-1}$. For $K\ge\ell\ge2$ and a real place $\tau$, the archimedean factor used for $\Pi(f)\times\Pi(h)^\vee$ is
\[
 \begin{split}
 L_\tau(w,D_{K-1}^\infty\times D_{\ell-1}^\infty)
 &=\Gamma_\C\left(w+\frac{K+\ell-2}{2}\right)
 \Gamma_\C\left(w+\frac{K-\ell}{2}\right),\\
 \Gamma_\C(z)&=2(2\pi)^{-z}\Gamma(z).
 \end{split}
\]
 At a ramified finite place the local factor is defined by the local Rankin--Selberg zeta integrals of \cite{JPSS}; no such finite place occurs for the representations used below.

The all-place and partial products are
\begin{equation*}
 \begin{split}
 L(w,\Sigma_1\times\Sigma_2)&=\prod_vL_v(w,\Sigma_{1,v}\times\Sigma_{2,v}),\\
 L^S(w,\Sigma_1\times\Sigma_2)&=\prod_{v\notin S}L_v(w,\Sigma_{1,v}\times\Sigma_{2,v}).
 \end{split}
\end{equation*}
They converge absolutely for $\Re(w)>1$. Their continuation and functional equation follow from \cite{JPSS}, and the pole criterion used below is \cite{CPS}.

\begin{theorem}[Cogdell--Piatetski-Shapiro \CPSthmRef]\label{thm:CPS}
Let $\Sigma_1$ and $\Sigma_2$ be irreducible cuspidal automorphic representations of $\GL_2(\A_F)$ with unitary central characters, and let $S$ contain the infinite places and every finite place at which either representation is ramified. If
\[
 \Sigma_1\not\simeq\Sigma_2\otimes|\det|^{it}\qquad\text{for every }t\in\R,
\]
then both $L(s,\Sigma_1\times\Sigma_2^\vee)$ and $L^S(s,\Sigma_1\times\Sigma_2^\vee)$ are entire. More generally, both are holomorphic at $s=1$ whenever $\Sigma_1\not\simeq\Sigma_2$; in the exceptional case the all-place product has a simple pole at $s=1$.
\end{theorem}
\begin{proof}
The assertions for the all-place product are \cite[Theorem~2.4]{CPS}. For the partial product, use
\[
 L^S(s,\Sigma_1\times\Sigma_2^\vee)=L(s,\Sigma_1\times\Sigma_2^\vee)
 \prod_{v\in S}L_v(s,\Sigma_{1,v}\times\Sigma_{2,v}^\vee)^{-1}.
\]
Each inverse local factor is entire and creates no pole.
\end{proof}

We also use the following form of Shahidi's theorem.
\begin{theorem}[Shahidi \ShahidiThmRef]\label{thm:Shahidi}
Let $\Sigma_1$ and $\Sigma_2$ be irreducible cuspidal globally generic automorphic representations of $\GL_2(\A_F)$ with unitary central characters. If $S$ contains every infinite place and every finite ramified place, then the partial product has no zeros on $\Re(s)=1$. In particular, wherever it is holomorphic there,
\begin{equation}\label{eq:shahidi-nonzero}
 L^S(1+it,\Sigma_1\times\Sigma_2)\in\C^\times\qquad(t\in\R).
\end{equation}
\end{theorem}

We write $|\det|^{it}(g)=|\det g|_{\A_F}^{it}$.
\begin{proposition}\label{prop:analytic-application}
Let $K>\ell\ge2$ be even, and let $f\in S_K(\omega_f)$ and $h\in S_\ell(\omega_h)$ be normalized full-level eigentuples. Then $\Pi(f)$ and $\Pi(h)^\vee$ are irreducible unitary cuspidal globally generic representations with unitary central characters $\omega_f$ and $\omega_h^{-1}$, respectively. They are unramified at every finite place, and
\begin{equation*}
 \Pi(f)_\tau\simeq D_{K-1}^\infty,\qquad\Pi(h)_\tau\simeq D_{\ell-1}^\infty
 \qquad(\tau\mid\infty).
\end{equation*}
Moreover,
\begin{equation}\label{eq:not-twist}
 \Pi(f)\not\simeq\Pi(h)\otimes|\det|^{it}\qquad(t\in\R).
\end{equation}
For $S=\Sigma_\infty$, the function $L^S(s,\Pi(f)\times\Pi(h)^\vee)$ is entire in $s$, and
\begin{equation}\label{eq:boundary-value}
 L^S(1,\Pi(f)\times\Pi(h)^\vee)\in\C^\times.
\end{equation}
\end{proposition}
\begin{proof}
Propositions~\ref{prop:local-components} and \ref{prop:global-properties} give all the stated properties of $\Pi(f)$ and $\Pi(h)^\vee$, including unramifiedness at every finite place. Hence $S=\Sigma_\infty$ contains every ramified place required by Theorems~\ref{thm:CPS} and \ref{thm:Shahidi}.

If the isomorphism excluded in \eqref{eq:not-twist} existed, then at any real place we would have
\[
 D_{K-1}^\infty\simeq D_{\ell-1}^\infty\otimes|\det|^{it}.
\]
The twist does not change the rotation exponents, while the least exponents are $K$ and $\ell$. This is impossible because $K>\ell$. Theorem~\ref{thm:CPS} therefore shows that $L^S(s,\Pi(f)\times\Pi(h)^\vee)$ is entire. Finally, Theorem~\ref{thm:Shahidi} shows that its value at $s=1$ is nonzero, proving \eqref{eq:boundary-value}.
\end{proof}

\section{Rankin--Selberg pairings and Hecke product identities}\label{sec:proofs}
In this section we prove Theorems~A and~B. We first compare the classical and automorphic Rankin--Selberg series, then treat the weight-two boundary, and finally apply the resulting nonvanishing to the product problem.

\subsection{The classical and automorphic Rankin--Selberg series}\label{subsec:series}
We first define the classical series. Let $K>\ell\ge2$ be even, $f\in S_K$, and $h\in S_\ell$. For $\Re(s)>(K+\ell)/2$, define the classical diagonal Rankin--Selberg series
\begin{equation}\label{eq:classical-series}
 \Lcl(s;f,h)=\sum_{0\ne\ideal\subset\OF}
 \frac{c(\ideal,f)\overline{c(\ideal,h)}}{N(\ideal)^s}.
\end{equation}
For a narrow ideal class character $\chi$, put
\begin{equation}\label{eq:finite-Hecke-L}
 L_F(s,\chi)=\prod_{0\ne\pp\subset\OF}
 \bigl(1-\chi(\pp)N(\pp)^{-s}\bigr)^{-1}
 \qquad(\Re(s)>1).
\end{equation}
Thus $L_F(s,1)=\zeta_F(s)$.

Suppose first that $f\in S_K(\omega_f)$ and $h\in S_\ell(\omega_h)$ are normalized eigentuples. If $H_j(x_1,x_2)=\sum_{m=0}^j x_1^m x_2^{j-m}$, then Proposition~\ref{prop:local-components}, the recurrence \eqref{eq:hecke-coefficients}, and Lemma~\ref{lem:satake-hecke} give
\begin{align*}
 c(\pp_v^j,f)&=q_v^{j(K-1)/2}H_j(\alpha_{1,f,v},\alpha_{2,f,v}),\\
 \overline{c(\pp_v^j,h)}&=q_v^{j(\ell-1)/2}
 H_j(\alpha_{1,h,v}^{-1},\alpha_{2,h,v}^{-1}).
\end{align*}
Here we used $|\alpha_{i,h,v}|=1$. From the elementary identity
\[
 \sum_{j\ge0}H_j(x_1,x_2)H_j(y_1,y_2)X^j
 =\frac{1-x_1x_2y_1y_2X^2}
 {\prod_{i,j=1}^2(1-x_iy_jX)},
\]
and \eqref{eq:satake-product}, we obtain, with $w=s-(K+\ell-2)/2$, the local identity
\begin{equation}\label{eq:local-comparison}
 \sum_{j\ge0}\frac{c(\pp_v^j,f)\overline{c(\pp_v^j,h)}}{q_v^{js}}
 =
 \frac{L_v(w,\Pi(f)_v\times\Pi(h)_v^\vee)}
 {L_{F,v}(2w,\omega_f\omega_h^{-1})},
\end{equation}
where $L_{F,v}(z,\chi)=(1-\chi(\pp_v)q_v^{-z})^{-1}$. Indeed, the numerator in the generating-function identity is $1-\omega_f(\pp_v)\omega_h(\pp_v)^{-1}q_v^{-2w}$.

\begin{proposition}\label{prop:classical-RS}
Let $w=s-(K+\ell-2)/2$.
If $f\in S_K(\omega_f)$ and $h\in S_\ell(\omega_h)$ are normalized eigentuples, then
\begin{equation}\label{eq:global-comparison}
 L_F(2w,\omega_f\omega_h^{-1})\Lcl(s;f,h)
 =L^S(w,\Pi(f)\times\Pi(h)^\vee),
\end{equation}
where $S=\Sigma_\infty$.

For arbitrary $f\in S_K$ and $h\in S_\ell$, the series \eqref{eq:classical-series} has a meromorphic continuation to $\C$ and is holomorphic in
\begin{equation}\label{eq:holomorphy-region}
 \Re(s)>\frac{K+\ell-1}{2}.
\end{equation}
\end{proposition}

\begin{proof}
Multiplying \eqref{eq:local-comparison} over all finite places gives \eqref{eq:global-comparison} for $\Re(w)>1$. Proposition~\ref{prop:analytic-application} shows that the Rankin--Selberg numerator is entire. Hence the quotient in \eqref{eq:global-comparison} gives a meromorphic continuation for each normalized eigenpair.

By the commutative normal Hecke theory recalled in Subsection~\ref{subsec:dictionary}, choose orthogonal bases $\{f_i\}$ of $S_K$ and $\{h_j\}$ of $S_\ell$ consisting of normalized eigentuples, with central characters $\omega_{f_i}$ and $\omega_{h_j}$. Write $f=\sum_iA_if_i$ and $h=\sum_jB_jh_j$.
By sesquilinearity and \eqref{eq:global-comparison},
\begin{equation}\label{eq:arbitrary-RS-decomposition}
 \Lcl(s;f,h)
 =\sum_{i,j}A_i\overline{B_j}\,
 \frac{L^S(w,\Pi(f_i)\times\Pi(h_j)^\vee)}
 {L_F(2w,\omega_{f_i}\omega_{h_j}^{-1})}.
\end{equation}
This finite sum is meromorphic on $\C$. If $\Re(2w)>1$, then each $L_F(2w,\omega_{f_i}\omega_{h_j}^{-1})$ is represented by an absolutely convergent nonzero Euler product. Since the Rankin--Selberg numerators are entire, every summand is holomorphic there. The condition $\Re(2w)>1$ is equivalent to $\Re(s)>(K+\ell-1)/2$, which proves \eqref{eq:holomorphy-region}.
\end{proof}

\subsection{The Petersson identity and the weight-two boundary}\label{subsec:boundary}
We introduce the Eisenstein family. Let $k,\ell\ge2$ be even and $r\ge0$, and put $K=k+\ell+2r$.
Let $\Gamma_{\infty,\lambda}$ be the stabilizer of infinity in $\Gamma_\lambda$. For
\begin{equation}\label{eq:eisenstein-region}
 \Re(u)>1-\frac k2,
\end{equation}
set
\begin{equation}\label{eq:eisenstein-family}
 E_{k,\lambda}(z,u)=q_\lambda^{-k/2-u}
 \sum_{\gamma\in\Gamma_{\infty,\lambda}\backslash\Gamma_\lambda}(Y^u\slashk{k}\gamma)(z),
 \quad E_k(\cdot,u)=(E_{k,\lambda}(\cdot,u))_\lambda.
\end{equation}
The series is absolutely and locally uniformly convergent in \eqref{eq:eisenstein-region}, also after any fixed number of $z$-derivatives. For $k>2$ put $E_k=E_k(\cdot,0)$. For $k=2$, its continuation is regular at $u=0$ and, since $d>1$, its value there is holomorphic in $z$. We write $E_2=E_2(\cdot,0)$. The continuation and the estimates needed below are recalled in Lemma~\ref{lem:fourier-bounds}.

For $u\ne0$, the components of $E_2(\cdot,u)$ are generally not holomorphic, but they are smooth weight-two functions satisfying
\[
 E_{2,\lambda}(\cdot,u)\slashk{2}\gamma=E_{2,\lambda}(\cdot,u)
 \qquad(\gamma\in\Gamma_\lambda).
\]
Let $\mathcal K$ be a finite set of cusp charts for all the components $\Gamma_\lambda\backslash\mathbb H^d$. For $\kappa\in\mathcal K$, let $\lambda(\kappa)$ be the component containing the cusp and choose $\sigma_\kappa\in\GL_2^+(F)$ carrying $\infty$ to that cusp. Put
\[
 E_{2,\kappa}(z,u)
   =\bigl(E_{2,\lambda(\kappa)}(\cdot,u)\slashk{2}\sigma_\kappa\bigr)(z),
 \qquad
 h_\kappa(z)
   =\bigl(h_{\lambda(\kappa)}\slashk{\ell}\sigma_\kappa\bigr)(z).
\]
Thus $\lambda$ indexes a narrow-class component, while $\kappa$ indexes a chosen cusp chart on one of the components.

Shimura's general Fourier theory for automorphic eigenfunctions gives a Fourier--Whittaker expansion in every cusp chart; see \cite[(2.12)--(2.21)]{Shimura1985}. Let $\mathfrak l_\kappa\subset F$ be the translation lattice of the chart and let $\mathfrak b_\kappa=\mathfrak l_\kappa^\vee$ be its trace-dual Fourier lattice. In our weight-two specialization
\[
 \sigma=2\one,\qquad \tau^{\mathrm{Sh}}=0,\qquad \rho=\one,
\]
Shimura's zero-mode equation \cite[(3.3)]{Shimura1985} and Propositions~3.1--3.2 there give the two exponents $u$ and $-1-u$ at $s=1+u$.

For $t>0$ and $\Re(\beta)>0$, put
\[
 V(t;\alpha,\beta)
 :=
 \frac{e^{-t/2}t^\beta}{\Gamma(\beta)}
 \int_0^\infty e^{-tv}(1+v)^{\alpha-1}v^{\beta-1}\,dv.
\]
For fixed $t>0$, this has a holomorphic continuation in $(\alpha,\beta)$ to $\C^2$. For $0\ne\xi\in\mathfrak b_\kappa$, write $\xi_\tau=\tau(\xi)$ and $t_\tau=4\pi|\xi_\tau|y_\tau$, and set
\[
 W_{\kappa,\xi}(y,u)=
 \prod_{\xi_\tau>0}V(t_\tau;2+u,u)
 \prod_{\xi_\tau<0}t_\tau^{-2}V(t_\tau;u,2+u).
\]
Using Shimura's Whittaker description \cite[(2.14)--(2.21), (4.21b)]{Shimura1985}, there are functions $A_\kappa(u)$, $B_\kappa(u)$ and scalar coefficients $c_{\kappa,\xi}(u)$ such that
\begin{equation}\label{eq:eisenstein-fourier}
 E_{2,\kappa}(z,u)
 =A_\kappa(u)Y^u+B_\kappa(u)Y^{-1-u}
 +\sum_{0\ne\xi\in\mathfrak b_\kappa}
 c_{\kappa,\xi}(u)W_{\kappa,\xi}(y,u)e^{2\pi i\Tr(\xi x)}.
\end{equation}

We shall also use a uniform moderate-growth estimate. By Proposition~4.4 and its proof in \cite{Shimura1985}, the growth condition \cite[(2.7c)]{Shimura1985} is uniform when the spectral parameter ranges over a compact set. Since $\sigma=2\one$, the factor $y^{\sigma/2}$ in Lemma~11.2 there is exactly $Y$. Applying that lemma and its proof, and using that only finitely many cusp charts occur, we may shrink the parameter disk and choose constants $B_0,C_0>0$, independent of $\kappa$, $u$, $x$, and $y$, such that
\begin{equation}\label{eq:uniform-moderate-growth}
 Y\,|E_{2,\kappa}(x+iy,u)|
 \le C_0\bigl(Y^{B_0}+Y^{-B_0}\bigr)
 \qquad(|u|\le2\delta_0)
\end{equation}
for some $\delta_0>0$ and all $x+iy\in\HH^d$. Here the constants depend only on the fixed field, the finitely many cusp charts, their normalizations, and the chosen compact parameter disk; in particular they do not depend on $x$, $y$, or $u$.

For a multi-index $a=(a_\tau)_\tau$, put
\[
 y^a=\prod_\tau y_\tau^{a_\tau},
 \qquad
 |\xi|_1=\sum_\tau|\xi_\tau|,
 \qquad
 T_\xi(y)=\sum_\tau|\xi_\tau|y_\tau.
\]
By a reduced cusp region at $\kappa$ we mean a standard cusp region with $Y\ge1$ in which the logarithmic unit directions are restricted to a fixed compact fundamental domain.

\begin{lemma}\label{lem:fourier-bounds}
There is $0<\delta<1/8$ such that the family and its Fourier coefficients in \eqref{eq:eisenstein-fourier} are holomorphic on a neighborhood of $|u|\le2\delta$. The following bounds hold for $j=0,1$ and $|u|\le\delta$, with constants independent of $u$, $\xi$, and $z$. They may be chosen uniformly over the finitely many cusps and components.

\noindent \textup{(i)} The constant coefficients satisfy
\begin{equation}\label{eq:constant-coeff-bounds}
 |\partial_u^jA_\kappa(u)|+|\partial_u^jB_\kappa(u)|\le C,
 \qquad B_\kappa(0)=0,\qquad |B_\kappa(u)|\le C|u|.
\end{equation}
\textup{(ii)} For some $M\ge0$, the nonzero scalar coefficients satisfy
\begin{equation}\label{eq:fourier-coeff-bounds}
 |\partial_u^jc_{\kappa,\xi}(u)|\le C(1+|\xi|_1)^M
 \qquad(0\ne\xi\in\mathfrak b_\kappa).
\end{equation}
\textup{(iii)} For every multi-index $a$, there are $M_a\ge0$ and $C_a>0$ such that, for every $y_\tau>0$,
\begin{equation}\label{eq:whittaker-derivative-bounds}
 y^a\left|\partial_u^jD^a
 \bigl(W_{\kappa,\xi}(y,u)e^{2\pi i\Tr(\xi x)}\bigr)\right|
 \le C_a e^{-\pi T_\xi(y)}
 \prod_\tau\bigl(1+(|\xi_\tau|y_\tau)^{-M_a}\bigr).
\end{equation}
\textup{(iv)} If $h\in S_\ell$, then, at every cusp and for every multi-index $b$ and $N>0$,
\begin{equation}\label{eq:cusp-decay}
 y^{(\ell/2)\one+b}|D^bh_\kappa(z)|
 \le C_{b,N}(1+Y)^{-N}
\end{equation}
on the reduced cusp region.
\end{lemma}
\begin{proof}
We treat the four assertions separately.

For (i), the continuation results of \cite[Theorem~4.2 and Propositions~3.1, 8.2(i),(iii)]{Shimura1985} show that the family is regular near $u=0$ and that the zero Fourier mode has exactly the two exponents appearing in \eqref{eq:eisenstein-fourier}. Since these exponents are $u$ and $-1-u$, they remain distinct on a sufficiently small disk. The constant Fourier coefficient is obtained by integrating $E_{2,\kappa}(x+iy,u)$ over the compact $x$-torus. Evaluating it at two fixed values of $Y$ therefore solves for $A_\kappa(u)$ and $B_\kappa(u)$ and shows that both are holomorphic in $u$. At $u=0$ the function $E_{2,\kappa}(z,0)$ is holomorphic in $z$, whereas $Y^{-1}$ is not; hence $B_\kappa(0)=0$. Cauchy's formula on a slightly larger disk now gives
\[
 |\partial_u^jA_\kappa(u)|+|\partial_u^jB_\kappa(u)|\le C
 \qquad(j=0,1),
\]
and the mean-value formula gives $|B_\kappa(u)|\le C|u|$. This proves (i).

For (ii), shrink $\delta$ if necessary so that $|u|\le2\delta$ lies in the disk of \eqref{eq:uniform-moderate-growth}. We extract the $\xi$-th Fourier coefficient at a height depending on $\xi$. By Shimura's uniform asymptotic \cite[(10.8)]{Shimura1985}, there is a sufficiently large fixed $T>0$ such that, for $|u|\le2\delta$, both $V(T;2+u,u)$ and $V(T;u,2+u)$ are nonzero and bounded away from zero in absolute value. Since there are only finitely many sign vectors, there is therefore a constant $c_T>0$ such that
\[
 |W_{\kappa,\xi}(y_\xi,u)|\ge c_T
 \qquad(|u|\le2\delta),
\]
where
\[
 y_{\xi,\tau}=\frac{T}{4\pi|\xi_\tau|}
 \qquad(\tau\in\Sigma_\infty).
\]
The constant $c_T$ is independent of $\kappa$, $\xi$, and $u$.

Let $\mathfrak l_\kappa$ be the translation lattice introduced above. Fourier coefficient extraction on the compact torus $F_\R/\mathfrak l_\kappa$ gives
\[
 \begin{split}
 &c_{\kappa,\xi}(u)W_{\kappa,\xi}(y_\xi,u)\\
 &\qquad=
 \frac{1}{\operatorname{vol}(F_\R/\mathfrak l_\kappa)}
 \int_{F_\R/\mathfrak l_\kappa}
 E_{2,\kappa}(x+iy_\xi,u)
 e^{-2\pi i\Tr(\xi x)}\,dx.
 \end{split}
\]
Hence \eqref{eq:uniform-moderate-growth} implies
\[
 |c_{\kappa,\xi}(u)|
 \le C\bigl(Y_\xi^{B_0-1}+Y_\xi^{-B_0-1}\bigr),
 \qquad
 Y_\xi=\left(\frac{T}{4\pi}\right)^d
 |N_{F/\Q}(\xi)|^{-1}.
\]
We may increase $B_0$ and assume $B_0\ge1$. Since $\mathfrak b_\kappa$ is a fixed fractional ideal, choose $m_\kappa\ge1$ with $m_\kappa\mathfrak b_\kappa\subset\OF$. Then, for $0\ne\xi\in\mathfrak b_\kappa$,
\[
 |N_{F/\Q}(\xi)|\ge m_\kappa^{-d}.
\]
It follows that
\[
 |c_{\kappa,\xi}(u)|
 \le C\bigl(1+|N_{F/\Q}(\xi)|\bigr)^{B_0+1}
 \le C'(1+|\xi|_1)^M
\]
for some $M\ge0$, where the last inequality follows from
$|N_{F/\Q}(\xi)|\le(|\xi|_1/d)^d$.

The coefficient extracted above is holomorphic in $u$. Moreover, by the choice of $T$, $W_{\kappa,\xi}(y_\xi,u)$ is holomorphic and nonzero on the larger disk $|u|\le2\delta$. Hence $c_{\kappa,\xi}(u)$ is holomorphic there. Cauchy's formula on the smaller disk $|u|\le\delta$ gives the same polynomial bound for $\partial_uc_{\kappa,\xi}(u)$. This proves \eqref{eq:fourier-coeff-bounds}.

For (iii), we use Shimura's estimate \cite[(10.11)]{Shimura1985}: uniformly for $(\alpha,\beta)$ in a compact set,
\begin{equation}\label{eq:V-basic-bound}
 |V(t;\alpha,\beta)|\le Ce^{-t/2}(1+t^{-B})
 \qquad(t>0).
\end{equation}
The integral defining $V$ gives, initially for $\Re\beta>0$ and then by analytic continuation,
\[
 t\partial_tV(t;\alpha,\beta)
 =\left(\beta-\frac t2\right)V(t;\alpha,\beta)
 -\beta V(t;\alpha,\beta+1).
\]
Iterating this identity shows that every fixed $t$-derivative is bounded by the right side of \eqref{eq:V-basic-bound}, with possibly a larger power of $t$ and $t^{-1}$. The extra factor $t^{-2}$ when $\xi_\tau<0$ only enlarges that power. Moreover, each application of $D_\tau$ to
$W_{\kappa,\xi}(y,u)e^{2\pi i\Tr(\xi x)}$ produces a factor $y_\tau^{-1}$ times a linear combination of such $t_\tau$-derivatives and of the exponential; after multiplying by $y^a$, all remaining factors are polynomial in the $t_\tau$. Absorbing these polynomials into $e^{-t_\tau/4}$ gives the exponential factor $e^{-\pi T_\xi(y)}$. Finally, the dependence on $u$ is holomorphic on the larger disk, so Cauchy's formula gives the same estimate after one $u$-derivative. This proves \eqref{eq:whittaker-derivative-bounds}.

For (iv), the Fourier expansion of the cusp form $h_\kappa$ has no constant term. Its derivatives have the form
\[
 D^bh_\kappa(z)
 =\sum_{\xi\gg0}
 \xi^b a_\kappa(\xi;h)e^{2\pi i\Tr(\xi x)}
 e^{-2\pi\Tr(\xi y)}.
\]
The Fourier coefficients have polynomial growth, whereas the exponential decays uniformly on a reduced cusp region. Consequently, for every $N>0$,
\[
 y^{(\ell/2)\one+b}|D^bh_\kappa(z)|
 \ll_{b,N}(1+Y)^{-N}.
\]
This is the standard rapid-decay estimate for cusp forms; see also
\cite[Theorem~2.24]{MS} for the general automorphic
rapid-decay theorem on Siegel sets.

We now complete the proof.
\end{proof}

We next fix reduced cusp neighborhoods.  For each component $\Gamma_\lambda\backslash\HH^d$, standard
reduction theory gives finitely many cusps; see
\cite[Chapter~I, Section~2]{Freitag}.  After choosing the cusp charts
$\kappa$ above, there is a height $Y_0\ge1$ such that, up to boundary
sets of measure zero,
\begin{equation}\label{eq:quotient-cusp-decomposition}
 \Gamma_\lambda\backslash\HH^d
 =\mathcal C_\lambda\ \cup\!
 \bigcup_{\lambda(\kappa)=\lambda}\mathcal U_\kappa(Y_0),
\end{equation}
where $\mathcal C_\lambda$ is compact and the sets
$\mathcal U_\kappa(Y_0)$ are pairwise disjoint cusp neighborhoods.
In the coordinates of the cusp chart $\kappa$, a reduced cusp
neighborhood is obtained by reducing the translation variable $x$
modulo the cusp lattice, reducing the logarithmic unit directions to a
fixed compact fundamental parallelepiped, and requiring $Y\ge Y_0$.

More explicitly, write
\[
 \eta_\tau=\log y_\tau-\frac1d\log Y,
 \qquad \sum_\tau\eta_\tau=0.
\]
The vector $(\eta_\tau)_\tau$ lies in the hyperplane
$\sum_\tau\eta_\tau=0$.  The logarithms of the totally positive units
form a lattice in this hyperplane, and our reduction chooses
$(\eta_\tau)_\tau$ in a fixed compact fundamental domain.  Hence there
is a constant $C_\kappa\ge1$ such that
\begin{equation}\label{eq:reduced-cusp-comparison}
 C_\kappa^{-1}Y^{1/d}
 \le y_\tau\le C_\kappa Y^{1/d}
 \qquad(\tau\in\Sigma_\infty)
\end{equation}
throughout $\mathcal U_\kappa(Y_0)$.  We shall use
\eqref{eq:quotient-cusp-decomposition} to estimate the Petersson norm
separately on the compact core and on the cusp neighborhoods.

\begin{lemma}\label{lem:boundary}
For every multi-index $a$, there is $C_a>0$ such that, with $\delta$ as in Lemma~\ref{lem:fourier-bounds},
\begin{equation}\label{eq:uniform-estimate}
 y^{\one+a}|\partial_uD^aE_{2,\kappa}(z,u)|
 \le C_a(1+Y)^2\qquad(|u|\le\delta)
\end{equation}
on every reduced cusp region. If $h\in S_\ell$, $r\ge0$, and
\[
 B_u=[E_2(\cdot,u),h]_r-[E_2,h]_r,
\]
then $B_u$ belongs to the Petersson $L^2$-space of weight $\ell+2r+2$ and
\[
 \|B_u\|_2=O(|u|).
\]
In particular, for every $f\in S_{\ell+2r+2}$,
\[
 \langle f,[E_2(\cdot,u),h]_r\rangle
 \longrightarrow
 \langle f,[E_2,h]_r\rangle
 \qquad(u\to0).
\]
\end{lemma}
\begin{proof}
We first prove \eqref{eq:uniform-estimate}.  Write
\[
 E_{2,\kappa}(z,u)
 =A_\kappa(u)Y^u+B_\kappa(u)Y^{-1-u}
 +\mathcal N_\kappa(z,u),
\]
where $\mathcal N_\kappa$ is the nonconstant Fourier part of
\eqref{eq:eisenstein-fourier}.

For the constant part, \eqref{eq:constant-coeff-bounds} gives, for every
fixed multi-index $a$,
\begin{equation}\label{eq:boundary-constant-part}
 y^{\one+a}
 \left|\partial_uD^a
 \bigl(A_\kappa(u)Y^u+B_\kappa(u)Y^{-1-u}\bigr)\right|
 \le C_a(1+Y)^2.
\end{equation}
Indeed, $D^aY^s$ is a fixed polynomial in $s$ times
$y^{-a}Y^s$, and one $u$-derivative introduces at most one factor
$\log Y$; for $|u|\le\delta$ and $Y\ge1$ these factors are absorbed by
$(1+Y)^2$.

For the nonconstant part, put
\[
 F_{\kappa,\xi}(z,u)
 =c_{\kappa,\xi}(u)W_{\kappa,\xi}(y,u)
  e^{2\pi i\Tr(\xi x)}.
\]
The coefficient estimate \eqref{eq:fourier-coeff-bounds}, the
Whittaker estimate \eqref{eq:whittaker-derivative-bounds}, and the
product rule give, for some $L\ge0$,
\begin{equation}\label{eq:boundary-single-fourier-term}
 \begin{split}
 &y^{\one+a}\Bigl(
 |D^aF_{\kappa,\xi}(z,u)|
 +|\partial_uD^aF_{\kappa,\xi}(z,u)|\Bigr)\\
 &\qquad\le
 CY(1+|\xi|_1)^L e^{-\pi T_\xi(y)}
 \prod_\tau\bigl(1+(|\xi_\tau|y_\tau)^{-L}\bigr).
 \end{split}
\end{equation}
On $\mathcal U_\kappa(Y_0)$, \eqref{eq:reduced-cusp-comparison} gives
\[
 T_\xi(y)=\sum_\tau|\xi_\tau|y_\tau
 \ge C_\kappa^{-1}Y^{1/d}|\xi|_1.
\]
Moreover, since $\mathfrak b_\kappa$ is a fixed fractional-ideal
lattice, there is $m_\kappa\ge1$ with
$m_\kappa\mathfrak b_\kappa\subset\OF$.  Thus for
$0\ne\xi\in\mathfrak b_\kappa$,
\[
 |N_{F/\Q}(\xi)|\ge m_\kappa^{-d},
 \qquad
 |\xi_\tau|^{-1}
 \le m_\kappa^d(1+|\xi|_1)^{d-1}.
\]
Together with $Y\ge1$ and \eqref{eq:reduced-cusp-comparison}, this
absorbs the product in \eqref{eq:boundary-single-fourier-term} into a
fixed polynomial in $1+|\xi|_1$.  Hence
\begin{equation}\label{eq:nonconstant-majorant}
 \begin{split}
 &y^{\one+a}\sum_{0\ne\xi\in\mathfrak b_\kappa}
 \Bigl(|D^aF_{\kappa,\xi}(z,u)|
 +|\partial_uD^aF_{\kappa,\xi}(z,u)|\Bigr)\\
 &\qquad\le
 CY\sum_{0\ne\xi\in\mathfrak b_\kappa}
 (1+|\xi|_1)^{L'}e^{-cY^{1/d}|\xi|_1}
 \le C'e^{-c'Y^{1/d}}.
 \end{split}
\end{equation}
For the last inequality, use that the nonzero lattice has a positive
minimum of $|\xi|_1$ and that the number of lattice points with
$n\le|\xi|_1<n+1$ is $O((n+1)^d)$.  Thus both the Fourier series and
its $u$-derivative converge normally for $|u|\le\delta$, so termwise
$u$-differentiation is valid.  Combining
\eqref{eq:boundary-constant-part} and
\eqref{eq:nonconstant-majorant} proves
\eqref{eq:uniform-estimate}.

Now put $\Delta_u=E_2(\cdot,u)-E_2$. For each cusp chart $\kappa$, let $\Delta_{u,\kappa}(z)=E_{2,\kappa}(z,u)-E_{2,\kappa}(z,0)$.
Holomorphy in $u$ and the fundamental theorem of calculus give
\begin{equation*}
 D^a\Delta_{u,\kappa}(z)
 =u\int_0^1\partial_uD^aE_{2,\kappa}(z,tu)\,dt,
\end{equation*}
so \eqref{eq:uniform-estimate} implies
\begin{equation}\label{eq:difference-estimate}
 y^{\one+a}|D^a\Delta_{u,\kappa}(z)|
 \le C_a|u|(1+Y)^2.
\end{equation}

Let
\[
 m=\ell+2r+2,
 \qquad
 B_u=[E_2(\cdot,u),h]_r-[E_2,h]_r.
\]
For a cusp chart $\kappa$ belonging to the component
$\lambda(\kappa)$, define the expression of $B_u$ in that chart by
\begin{equation*}
 B_{u,\kappa}
 :=B_{u,\lambda(\kappa)}\slashk{m}\sigma_\kappa.
\end{equation*}
Rankin--Cohen covariance gives
\begin{equation*}
 B_{u,\kappa}
 =[E_{2,\kappa}(\cdot,u),h_\kappa]_r
  -[E_{2,\kappa}(\cdot,0),h_\kappa]_r
 =[\Delta_{u,\kappa},h_\kappa]_r.
\end{equation*}
Every term in the last bracket is a constant multiple of
$(D^a\Delta_{u,\kappa})(D^{r\one-a}h_\kappa)$.  Note that 
$(\one+a)+\bigl((\ell/2)\one+r\one-a\bigr)
 =\frac m2\one. $
\eqref{eq:difference-estimate} and the cusp-form decay 
\eqref{eq:cusp-decay} then imply, for every $N>0$,
\begin{equation*}
 Y^{m/2}|B_{u,\kappa}(z)|
 \le C_N|u|(1+Y)^{2-N}
 \qquad(z\in\mathcal U_\kappa(Y_0)).
\end{equation*}
Therefore, choosing $N>2$,
\begin{equation}\label{eq:cusp-L2-bound}
 \int_{\mathcal U_\kappa(Y_0)}
 |B_{u,\kappa}(z)|^2Y^m\,d\mu(z)
 \le C|u|^2
 \int_{\mathcal U_\kappa(Y_0)}(1+Y)^{4-2N}\,d\mu(z)
 \ll |u|^2.
\end{equation}
The last integral is finite because the cusp neighborhood has finite
invariant volume and $(1+Y)^{4-2N}\le1$ for $N>2$.

It remains to estimate the compact core. Choose compact sets of representatives
\(\widetilde{\mathcal C}_\lambda\subset\HH^d\) for the compact subsets
\(\mathcal C_\lambda\) in \eqref{eq:quotient-cusp-decomposition}. Since
\(B_{u,\lambda}(z)\) is smooth in \(z\), holomorphic in \(u\) near \(u=0\),
and \(B_{0,\lambda}=0\), compactness gives

$$
 \sup_{z\in\widetilde{\mathcal C}_\lambda}
 |B_{u,\lambda}(z)|\ll |u|
 \qquad (|u|\le\delta).
$$

Moreover, \(Y^m\) is bounded on each
\(\widetilde{\mathcal C}_\lambda\), and the compact core has finite volume.
Hence
\begin{equation}\label{eq:compact-L2-bound}
\sum_{\lambda=1}^{h_F^+}
\int_{\mathcal C_\lambda}
|B_{u,\lambda}(z)|^2Y^m,d\mu(z)
\ll |u|^2.
\end{equation}

Finally, the Petersson norm is invariant under the cusp changes of
coordinates, and \eqref{eq:quotient-cusp-decomposition} gives, up to
sets of measure zero,
\begin{equation*}
 \begin{split}
 \|B_u\|_2^2
 ={}&\sum_{\lambda=1}^{h_F^+}
 \int_{\mathcal C_\lambda}
 |B_{u,\lambda}(z)|^2Y^m\,d\mu(z)\\
 &+\sum_{\kappa}
 \int_{\mathcal U_\kappa(Y_0)}
 |B_{u,\kappa}(z)|^2Y^m\,d\mu(z).
 \end{split}
\end{equation*}
Combining \eqref{eq:cusp-L2-bound} and
\eqref{eq:compact-L2-bound}, and using that there are only finitely
many cusp charts, yields
\[
 \|B_u\|_2^2\ll |u|^2,
 \qquad
 \|B_u\|_2\ll |u|.
\]
Finally, Cauchy--Schwarz gives
\[
 |\langle f,B_u\rangle|
 \le\|f\|_2\|B_u\|_2\longrightarrow0.
\]
\end{proof}

We unfold the Petersson pairing. For $k,\ell>2$, the corresponding
arbitrary narrow class number calculation is
\cite[Proposition~4.4 and Corollary~4.5]{MZZ}; for $h_F^+=1$ and
$k\ge4$, see also \cite[Lemma~3.2]{ZZ}. To treat $k=2$, we unfold
the Eisenstein family in the half-plane $\Re(u)>0$ and then let
$u\to0$. The class factors cancel in the calculation below. For $(u)_0^\downarrow=1$ and
\[
 (u)_t^\downarrow=u(u-1)\cdots(u-t+1)\qquad(t>0),
\]
put
\begin{equation}\label{eq:gamma-factor}
 G_{k,\ell,r}(u)=\frac1{(4\pi)^{K-1+u}}\sum_{t=0}^r
 \binom{k+r-1}{r-t}\binom{\ell+r-1}{t}(u)_t^\downarrow\Gamma(K-1+u-t).
\end{equation}
For a multi-index $t=(t_\tau)_\tau$, we also use
\begin{equation*}
 C_{k,\ell,r}(t)=\prod_\tau\binom{k+r-1}{r-t_\tau}\binom{\ell+r-1}{t_\tau},
 \qquad(u)_t^\downarrow=\prod_\tau(u)_{t_\tau}^\downarrow.
\end{equation*}

\begin{theorem}[Theorem A]\label{thm:pairing}
Let $f\in S_K$ and $h\in S_\ell$. For $\Re(u)>1-k/2$,
\begin{equation}\label{eq:deformed-identity}
 \langle f,[E_k(\cdot,\overline u),h]_r\rangle
 =
 D_F^{-1/2}G_{k,\ell,r}(u)^d
 \Lcl(K-r-1+u;f,h).
\end{equation}
For $k=2$, this identity extends to $u=0$.

If $f\in S_K(\omega_f)$ and $h\in S_\ell(\omega_h)$ are normalized eigentuples, then
\begin{equation}\label{eq:petersson-automorphic}
 \langle f,[E_k,h]_r\rangle
 =
 \CF\,
 \frac{L^S(k/2,\Pi(f)\times\Pi(h)^\vee)}
      {L_F(k,\omega_f\omega_h^{-1})}
 \ne0.
\end{equation}
\end{theorem}
\begin{proof}
The conjugate parameter is used because the Petersson pairing is
conjugate-linear in its second variable. After complex conjugation in
the integral, the left-hand side of \eqref{eq:deformed-identity} is
holomorphic in $u$, and the unfolded seed is $Y^u$.

Ramanujan--Petersson, applied after expansion in cuspidal eigenbases,
gives, for every $\varepsilon>0$,
\[
 |c(\ideal,f)c(\ideal,h)|
 \ll_\varepsilon
 N(\ideal)^{(K+\ell-2)/2+\varepsilon}.
\]
Note that $\Re(K-r-1+u)>(K+\ell)/2$ is equivalent to $\Re(u)>1-k/2$.
The series $\Lcl(K-r-1+u;f,h)$ converges absolutely in this
half-plane. The Poincar\'e series \eqref{eq:eisenstein-family},
together with the required $z$-derivatives, is also absolutely
convergent there. The rapid decay of the cusp forms and their
derivatives therefore justifies the unfolding and all termwise
integrations below.

Put $F_{\mathbb R}=F\otimes_{\mathbb Q}\mathbb R$ and $L_\lambda=\mathfrak t_\lambda^{-1}\mathfrak d_F^{-1}$.
We first describe the quotient by the stabilizer of infinity. If
\[
 \gamma=
 \begin{pmatrix}
  a&b\\0&d
 \end{pmatrix}
 \in\Gamma_{\infty,\lambda},
\]
then $ad\in\mathcal O_F^{\times,+}$. Hence
$a,d\in\mathcal O_F^\times$, and, modulo the scalar matrix $dI_2$,
\[
 \gamma\sim
 \begin{pmatrix}
  \varepsilon&\beta\\0&1
 \end{pmatrix},
 \qquad
 \varepsilon=\frac ad\in\mathcal O_F^{\times,+},
 \qquad
 \beta=\frac bd\in L_\lambda.
\]
Scalar matrices act trivially on $\mathbb H^d$, so the effective
action of $\Gamma_{\infty,\lambda}$ is generated by
\[
 z\longmapsto z+\beta\quad(\beta\in L_\lambda),
 \qquad
 z\longmapsto\varepsilon z
 \quad(\varepsilon\in\mathcal O_F^{\times,+}).
\]
Since $\varepsilon L_\lambda=L_\lambda$, choose a fundamental
parallelepiped for $L_\lambda$ in $F_{\mathbb R}$ and a fundamental
domain $\mathcal F_\lambda$ for the action of
$\mathcal O_F^{\times,+}$ on $\mathbb R_{>0}^d$. Up to boundary
sets of measure zero,
\[
 (F_{\mathbb R}/L_\lambda)\times\mathcal F_\lambda
\]
is a fundamental domain for
$\Gamma_{\infty,\lambda}\backslash\mathbb H^d$.

We now unfold the Poincar\'e series
\eqref{eq:eisenstein-family}. Rankin--Cohen covariance replaces the
integral over $\Gamma_\lambda\backslash\mathbb H^d$ together with
the sum over
$\Gamma_{\infty,\lambda}\backslash\Gamma_\lambda$ by an integral
over $\Gamma_{\infty,\lambda}\backslash\mathbb H^d$. Hence
\[
\begin{split}
 &\langle f,[E_k(\cdot,\overline u),h]_r\rangle\\
 &=
 \sum_{\lambda=1}^{h_F^+}
 \sum_{t\in\{0,\ldots,r\}^{\Sigma_\infty}}
 q_\lambda^{-k/2-u}
 C_{k,\ell,r}(t)
 \frac{(u)_t^\downarrow}{(4\pi)^{|t|}}\\
 &\qquad\times
 \int_{(F_{\mathbb R}/L_\lambda)\times\mathcal F_\lambda}
 f_\lambda(z)\,
 \overline{D^{r\one-t}h_\lambda(z)}
 \prod_\tau y_\tau^{K+u-t_\tau-2}\,dx\,dy.
\end{split}
\]
Indeed,
\[
 D_\tau^{t_\tau}y_\tau^u
 =
 (-1)^{t_\tau}
 \frac{(u)_{t_\tau}^\downarrow}{(4\pi)^{t_\tau}}
 y_\tau^{u-t_\tau},
\]
and these signs cancel those in the Rankin--Cohen bracket.

Insert the Fourier expansions of $f_\lambda$ and $h_\lambda$.
The $x$-integration gives
\[
 \int_{F_{\mathbb R}/L_\lambda}
 e^{2\pi i\Tr((\xi-\eta)x)}\,dx
 =
 D_F^{-1/2}q_\lambda\,\delta_{\xi,\eta},
\]
since the dual lattice of $L_\lambda$ is $\mathfrak t_\lambda$.
The simultaneous action of $\mathcal O_F^{\times,+}$ on the
surviving index $\xi$ and on $y$ unfolds $\mathcal F_\lambda$ to
$\mathbb R_{>0}^d$. For each real place, the remaining Mellin
integral is
\[
 \frac{\xi^{r-t}}{(4\pi)^t}
 \int_0^\infty
 e^{-4\pi\xi y}y^{K+u-t-2}\,dy
 =
 \frac{\Gamma(K-1+u-t)}
 {(4\pi)^{K-1+u}}
 \xi^{-(K-r-1+u)}.
\]
The sum over the multi-index $t$ is therefore the product, over the
$d$ real places, of the one-variable sum in
\eqref{eq:gamma-factor}. We obtain
\begin{equation}\label{eq:classwise-sum}
\begin{split}
 &\langle f,[E_k(\cdot,\overline u),h]_r\rangle\\
 &\quad=
 D_F^{-1/2}G_{k,\ell,r}(u)^d
 \sum_{\lambda=1}^{h_F^+}q_\lambda^{1-k/2-u}
 \sum_{\substack{
   \xi\in\mathfrak t_\lambda/
        \mathcal O_F^{\times,+}\\
   \xi\gg0}}
 \frac{
  a_\lambda(\xi;f)\overline{a_\lambda(\xi;h)}
 }{N(\xi)^{K-r-1+u}}.
\end{split}
\end{equation}

For
$\ideal=(\xi)\mathfrak t_\lambda^{-1}$,
\[
 N(\ideal)=q_\lambda N(\xi),
 \qquad
 a_\lambda(\xi;f)\overline{a_\lambda(\xi;h)}
 =
 q_\lambda^{-(K+\ell)/2}
 c(\ideal,f)\overline{c(\ideal,h)}.
\]
Note that $1-k/2-(K+\ell)/2=-(K-r-1)$.
Thus the total power of $q_\lambda$ in each term of
\eqref{eq:classwise-sum} is
$q_\lambda^{-(K-r-1+u)}$. Every nonzero integral ideal occurs
exactly once, and hence
\eqref{eq:deformed-identity} follows. In particular, no factor
$h_F^+$ or $(h_F^+)^{-1}$ occurs.

If $k>2$, then $u=0$ lies in the half-plane
$\Re(u)>1-k/2$, so we may set $u=0$ directly.
Suppose now that $k=2$. Then $u=0$ is the boundary of this
half-plane. Lemma~\ref{lem:boundary} gives
\[
 \|[E_2(\cdot,\overline u),h]_r-[E_2,h]_r\|_2
 \longrightarrow0
 \qquad(u\to0),
\]
and therefore
\[
 \langle f,[E_2(\cdot,\overline u),h]_r\rangle
 \longrightarrow
 \langle f,[E_2,h]_r\rangle.
\]

On the other hand, we have
\[
 K-r-1
 =
 \ell+r+1
 >
 \ell+r+\frac12
 =
 \frac{K+\ell-1}{2}.
\]
Hence Proposition~\ref{prop:classical-RS} shows that
$\Lcl(K-r-1+u;f,h)$ is holomorphic at $u=0$. The function
$G_{2,\ell,r}(u)$ is holomorphic there as well. Thus
\eqref{eq:deformed-identity} extends to $u=0$.

Note that
\[
 G_{k,\ell,r}(0)
 =
 \binom{k+r-1}{r}
 \frac{\Gamma(K-1)}{(4\pi)^{K-1}}
 =
 \binom{k+r-1}{r}
 \frac{(K-2)!}{(4\pi)^{K-1}}.
\]
We now deduce
\[
 \langle f,[E_k,h]_r\rangle
 =
 \CF\,\Lcl(K-r-1;f,h).
\]

Assume now that $f\in S_K(\omega_f)$ and $h\in S_\ell(\omega_h)$ are normalized eigentuples. In \eqref{eq:global-comparison}, put $s=K-r-1$ and $w=s-(K+\ell-2)/2=k/2$.
We obtain
\[
 L_F(k,\omega_f\omega_h^{-1})\Lcl(K-r-1;f,h)
 =L^S(k/2,\Pi(f)\times\Pi(h)^\vee),
\]
and therefore \eqref{eq:petersson-automorphic}.

For $k>2$, both Euler products in \eqref{eq:petersson-automorphic} are absolutely convergent and nonzero. For $k=2$, Proposition~\ref{prop:analytic-application} gives
\[
 L^S(1,\Pi(f)\times\Pi(h)^\vee)\in\C^\times,
\]
while $L_F(2,\omega_f\omega_h^{-1})$ is absolutely convergent and nonzero. Hence the Petersson pairing is nonzero in this case as well. This proves the theorem.
\end{proof}

\begin{remark}
For $k>2$, M.~Zhang and Y.~Zhang proved the identity for arbitrary level and arbitrary narrow class number in \cite[Proposition~4.4 and Corollary~4.5]{MZZ}. 
\end{remark}
\begin{corollary}\label{cor:weight-two}
Let $h\in S_\ell(\omega_h)$ and $f\in S_{\ell+2r+2}(\omega_f)$ be normalized eigentuples. Then
\begin{equation*}
 \langle f,[E_2,h]_r\rangle
 =D_F^{-1/2}\left(\frac{(r+1)(\ell+2r)!}{(4\pi)^{\ell+2r+1}}\right)^d
 \frac{L^S(1,\Pi(f)\times\Pi(h)^\vee)}
 {L_F(2,\omega_f\omega_h^{-1})}\ne0.
\end{equation*}
\end{corollary}
\begin{proof}
This is \eqref{eq:petersson-automorphic} with $k=2$. Proposition~\ref{prop:analytic-application} gives $L^S(1,\Pi(f)\times\Pi(h)^\vee)\in\C^\times$, while $L_F(2,\omega_f\omega_h^{-1})$ is represented by an absolutely convergent nonzero Euler product.
\end{proof}

\subsection{Non-eigenform criterion and product classification}\label{subsec:classification}
We apply Theorem~\ref{thm:pairing} and Corollary~\ref{cor:weight-two} to the product problem considered in \cite{MZZ,ZZ,HQZ}. By a product identity we mean an equality
\[
 P=cGH,
\]
where $P,G,H$ are nonzero normalized simultaneous Hecke eigenforms, $c\in\C^\times$, and the weight of $P$ is the sum of the weights of $G$ and $H$. Equivalently, the pointwise product $GH$ is a Hecke eigenform, and $c$ is chosen so that $c(\OF,P)=1$.

For even $k>2$ and arbitrary narrow class number, Shimura constructs
a Hecke-normalized Eisenstein eigentuple, which we denote by
$E_k^{\mathrm{eig}}$, characterized by
\[
 c(\mathfrak a,E_k^{\mathrm{eig}})
 =
 \sigma_{k-1}(\mathfrak a)
 :=
 \sum_{\mathfrak r\mid\mathfrak a}N(\mathfrak r)^{k-1},
 \qquad
 c(\OF,E_k^{\mathrm{eig}})=1;
\]
see \cite{Shimura1978}. This eigentuple should be distinguished from
the Poincar\'e-normalized Eisenstein tuple $E_k$ used in
Theorem~\ref{thm:pairing}. When $h_F^+=1$, comparison of the Fourier
coefficients gives
\begin{equation}\label{eq:Eisenstein-normalization}
 E_k^{\mathrm{eig}}
 =
 2^{-d}\zeta_F(1-k)E_k.
\end{equation}
For $k=2$ and $h_F^+=1$, we use the same normalization and set
\[
  E_2^{\mathrm{eig}}
  :=2^{-d}\zeta_F(-1)E_2.
\]
Thus $c(\OF,E_2^{\mathrm{eig}})=1$.

\begin{corollary}\label{cor:projected-noneigenform}
Let $h\in S_\ell(\omega_h)$ be a normalized full-level cuspidal eigentuple and let $\chi$ be a narrow ideal class character. Put $K=k+\ell+2r$. If $\dim S_K(\chi\omega_h)>1$, then $[P_\chi E_k,h]_r$ is nonzero and is not a simultaneous Hecke eigentuple.
\end{corollary}
\begin{proof}
By \eqref{eq:projector-bracket},
\[
 [P_\chi E_k,h]_r=P_{\chi\omega_h}[E_k,h]_r
 \in S_K(\chi\omega_h).
\]
For every normalized eigentuple $f\in S_K(\chi\omega_h)$, Theorem~\ref{thm:pairing} gives
\[
 \langle f,[P_\chi E_k,h]_r\rangle
 =\langle f,[E_k,h]_r\rangle
 =\CF
 \frac{L^S(k/2,\Pi(f)\times\Pi(h)^\vee)}{L_F(k,\chi)}\ne0.
\]
Thus the bracket is nonzero. If it were an eigentuple, the commuting normal Hecke algebra would provide a normalized eigentuple $f\in S_K(\chi\omega_h)$ orthogonal to it, contradicting the displayed formula.
\end{proof}

\begin{corollary}\label{cor:noneigenform}
Let $h\in S_\ell(\omega_h)$ be a normalized full-level cuspidal eigentuple. Put $K=k+\ell+2r$. If there is a narrow ideal class character $\chi$ such that $\dim S_K(\chi\omega_h)>1$, then $[E_k,h]_r$ is not a simultaneous Hecke eigentuple. In particular, if $h_F^+=1$ and $\dim S_K>1$, then $[E_k,h]_r$ is not a Hecke eigenform.
\end{corollary}
\begin{proof}
Suppose that $[E_k,h]_r$ were a simultaneous Hecke eigentuple. Choose $\chi$ with $\dim S_K(\chi\omega_h)>1$. By \eqref{eq:projector-bracket},
\[
 [P_\chi E_k,h]_r=P_{\chi\omega_h}[E_k,h]_r.
\]
The left-hand side is nonzero by Corollary~\ref{cor:projected-noneigenform}; since $P_{\chi\omega_h}$ commutes with every $T_v$, it would therefore be a simultaneous Hecke eigentuple, contradicting the same corollary. If $h_F^+=1$, the only narrow ideal class character is trivial and $S_K(\omega_h)=S_K$, which gives the final assertion.
\end{proof}

When $h_F^+=1$, taking $r=0$ in Corollary~\ref{cor:noneigenform} gives $\dim S_{k+\ell}>1\Longrightarrow \Eeig_kh$ is not an eigenform.

\begin{proposition}\label{prop:HQZ-reductions}
The conclusions of \cite[Propositions~4.12 and~4.13]{HQZ} remain valid without GRH. Thus, over a real quadratic field of narrow class number one and discriminant $D_F>5$, no full-level product identity $P=\Eeig_2h$ exists with $h$ cuspidal.
\end{proposition}
\begin{proof}
In \cite{HQZ}, the only step in these two propositions that is not already unconditional is the use of \cite[Theorem~3]{HQZ} when the Eisenstein weight is $2$. Corollary~\ref{cor:noneigenform} gives the required dimension obstruction without GRH.

In the inert case, the coefficient bounds and the reductions leading to \cite[Table~2]{HQZ} are unconditional. For every remaining triple with target dimension greater than one, replace the invocation of \cite[Theorem~3]{HQZ} by Corollary~\ref{cor:noneigenform}. The only triple left is $(k,\ell,D_F)=(2,2,13)$, which would require a nonzero form in $S_2$; Ishikawa's table gives $S_2=0$ for $D_F=13$ \cite[p.~85]{Ishikawa}.

In the non-inert case, the bounds and \cite[Table~3]{HQZ} are again unconditional. When the Eisenstein weight is $4$, the target spaces in the remaining range have dimension greater than one, and Corollary~\ref{cor:noneigenform} gives the required exclusion. When the Eisenstein weight is $2$, only $(k,\ell,D_F)=(2,2,8)$ remains, but Ishikawa gives $S_2=0$ for $D_F=8$ \cite[p.~87]{Ishikawa}. Hence both propositions hold without GRH.
\end{proof}

\begin{theorem}[Theorem B]\label{thm:classification}
Among normalized full-level Hilbert Hecke eigenforms of even parallel weights at least two over real quadratic fields of narrow class number one, the only product identities are
\begin{equation}\label{eq:classification}
 \begin{gathered}
 F=\Q(\sqrt5),\qquad \Eeig_4=60(\Eeig_2)^2,\qquad h_8=120\Eeig_2h_6,\\
 c(\OF,\Eeig_k)=1.
 \end{gathered}
\end{equation}
Here $h_6$ and $h_8$ are the unique normalized cusp eigenforms of weights $6$ and $8$ over $\Q(\sqrt5)$.
\end{theorem}
\begin{proof}
The Eisenstein--Eisenstein case and the Eisenstein--cusp case with
Eisenstein weight at least four are unconditional by
\cite[Theorem~1]{HQZ}.

A cusp--cusp product is impossible. Indeed, if $g$ and $h$ are cusp
forms, then $c(\OF,gh)=0$. A contribution to this coefficient would
require
\[
 1=\xi+\eta,\qquad
 \xi,\eta\in\OF,\qquad
 \xi,\eta\gg0.
\]
Then $0<\tau(\xi)<1$ at every real embedding $\tau$, and hence
$0<N(\xi)<1$, contradicting $N(\xi)\in\Z_{>0}$. Thus a cusp--cusp
product cannot equal a normalized eigenform, whose unit-ideal
coefficient is $1$.

It remains only to consider the weight-two Eisenstein--cusp case.
Proposition~\ref{prop:HQZ-reductions} excludes it over every real
quadratic field with $D_F>5$. Hence no further product identity occurs
there. Over $\Q(\sqrt5)$, Joshi and Y.~Zhang proved that the two
identities in \eqref{eq:classification} are the complete list
\cite[Theorem~7.4]{JZ}.
\end{proof}

\section*{Acknowledgments}
The author is grateful to Haowu Wang for his valuable guidance and detailed suggestions throughout the preparation of this manuscript.

\section*{Declaration of AI use}
The author used ChatGPT (OpenAI) for assistance with language editing and for discussing and checking mathematical arguments during the preparation of this manuscript. All mathematical statements, arguments, and proofs included in the final version were independently verified by the author, who takes full responsibility for the content of the article.

\end{document}